\documentclass[a4paper,english,11pt]{article} 
\usepackage{graphicx}
\usepackage{amsmath}
\usepackage{amssymb}
\usepackage[utf8]{inputenc}
\usepackage[T1]{fontenc}
\usepackage[english]{babel}
\usepackage[all]{xy}
\usepackage[top=25mm, bottom=40mm, left=30mm , right=30mm]{geometry}
\usepackage{amssymb,amsmath,amscd,amsfonts,amsthm,mathrsfs}
\usepackage{geometry}
\usepackage{times}  
\usepackage{fancyhdr}
\usepackage{pinlabel}
\usepackage{array}

\newcommand{\R}{\mathbf R} 
\newcommand{\N}{\mathbf N}
\newcommand{\Z}{\mathbf Z}

\newtheorem{theorem}{Theorem}[section]   
\newtheorem{definition}[theorem]{Definition}
\newtheorem{example}[theorem]{Example}
\newtheorem{corollary}[theorem]{Corollary}
\newtheorem{proposition}[theorem]{Proposition}

\newtheorem{lemma} [theorem]{Lemma}
\newtheorem{remark}[theorem]{Remark}
\newtheorem{fact}[theorem]{Fact}

\begin{document}
\title{On the existence and stability of two-dimensional \\Lorentzian tori without conjugate points}
\author{Lilia Mehidi}
\maketitle
\begin{abstract}
Infinitely many new examples of compact Lorentzian surfaces {\em without conjugate points} are given. Further, we study the existence and the stability of this property among Lorentzian metrics with a Killing field. We obtain a new obstruction and prove that the Clifton- Pohl torus and some of our examples are as stable as possible. This shows that in constrast with the Riemannian Hopf theorem, the absence of conjugate points in the Lorentzian setting is neither "special" nor rigid.
\end{abstract}
\section{Introduction} 
The absence of conjugate points on Riemannian tori has rigid effects on the metric structure. A result by E. Hopf in 1948 for a two-dimensional torus \cite{2}, and by Burago and Ivanov in 1994 in any dimension \cite{7} states that any Riemannian torus with no conjugate points is necessarily flat. However, it appears that the Hopf theorem does not hold in the Lorentzian setting; in fact, Bavard and Mounoud proved in \cite{9} that the so called Clifton-Pohl torus (see Equation (\ref{Tore de Clifton Pohl})) has no conjugate points. 
The Clifton-Pohl torus and its few natural deformations (see below) are the only known examples of Lorentzian metrics on the torus without conjugate points. Recall that any compact connected Lorentzian surface is homeomorphic to the torus or the Klein bottle. In this work, we give infinitely many new examples of geometrically non-equivalent Lorentzian tori and Klein bottles without conjugate points; moreover, we prove that some of them (including the Clifton-Pohl torus) admit a large space of deformations among metrics without conjugate points.	\\

Given a non-flat torus $T$ with a non-trivial Killing field $K$, the flow of $K$ induces a free action of the group $S^1$ on $T$ (see \cite{10}, Theorem 3.25). The orbits of $K$ are therefore periodic with the same period. Given the Clifton-Pohl metric, a first attempt to obtain deformations of such a metric without conjugate points can be achieved in three different ways:  varying the period of the Killing field;  acting by an homothety on the torus; or  acting by a "twist" along a (closed) orbit of $K$. Although the variations above give non-isometric metrics on the torus without conjugate points, these examples are all locally "the same", having all the same universal cover (up to homothety).
Less trivial deformations of the Clifton-Pohl torus without conjugate points are obtained by Mounoud in \cite{11}, as metrics projectively equivalent to the Clifton-Pohl torus; this gives a $2$-dimensional family of Lorentzian tori without conjugate points, with non-isometric universal cover. \\
 
When $K$ is timelike or spacelike, a result of Gutierrez, Palomo and Romero in \cite{8} shows that if the surface does not have conjugate points, it must be flat. In this paper, we suppose that $K$ has a null orbit.  Let $(\widetilde{T},\widetilde{K})$ be the universal cover of a non-flat torus $T$ with Killing field $K$. The action  of $K$ on the torus given above allows to define global coordinates $(u,\theta)$ on the universal cover $\widetilde{T}$ on which $\widetilde{K}$ is given by $\partial_u$. In this way, one can define $\widetilde{\kappa}(\theta)$ as the sectional curvature defined on the space of leaves of $\widetilde{K}$ (which is a line).   We prove the following
\begin{theorem}\label{famille SPC, introduction}
Let $(T,K)$ be a Lorentzian torus with a Killing field, and let  $(\widetilde{T},\widetilde{K})$ be its universal cover. Suppose that\\
	(i) the null orbits of $\widetilde{K}$ are geodesically incomplete,\\
	(ii) there is only one critical orbit of $\widetilde{K}$ in each band of $\widetilde{T}$,\\
	(iii) the curvature $\widetilde{\kappa}$ is a monotone function between two consecutive critical orbits of $\widetilde{K}$,\\ 
	(iv) the reflections with respect to the critical orbits of $\widetilde{K}$ act on $\widetilde{T}$ by isometry, \\
	(v) the foliation orthogonal to $\widetilde{K}$ has only Reeb components.\\
	Then $T$ has no conjugate points.   
\end{theorem}
Here, {\em an open band} is a connected component of the set $\{ \langle \widetilde{K},\widetilde{K} \rangle \neq 0\}$ in $\widetilde{T}$, and a critical orbit of $\widetilde{K}$ is an orbit corresponding to the critical points of the function $\langle \widetilde{K},\widetilde{K} \rangle$ (these orbits are geodesics). Many explicit new (analytic) examples of Lorentzian tori without conjugate points, as well as Klein bottles with the same property, will be deduced from Theorem \ref{famille SPC, introduction} in a very simple way. \\
The assumption (i) implies that the null orbits of $\widetilde{K}$ are isolated and that the norm of $\widetilde{K}$ changes sign when crossing such an orbit transversally (please see \cite{10}, Lemma 2.25). In particular, there exists only finitely many null orbits of $K$ in $T$. 	

\begin{remark}
Each of these examples can be deformed to families of metrics without conjugate points, in the same way as previously done with the Clifton-Pohl torus (changing the period of the orbits of $K$, changing by homothety, twist or by projective deformation). 
\end{remark}

The structure of compact Lorentzian surfaces with a one parameter group of isometries is already studied in \cite{10}. Although the property of being without conjugate points doesn't appear to be strong enough to expect a rigidity phenomenon in this subclass of Lorentzian surfaces, Theorem 5.29, \cite{10}, gives  obstructions for such tori to be without conjugate points. In particular, it follows from this theorem that a Lorentzian torus with a Killing field, without conjugate points, is either flat or non-homotopic to the flat metric. Assuming condition (i) in Theorem \ref{famille SPC, introduction} (which is in some sense generic), we give a new obstruction for those tori to be without conjugate points. It is easily seen that a geodesic parametrization of an incomplete null orbit of $K$ is given by $\frac{1}{\lambda} e^{\lambda t}$, where $K=\partial_t$, and $\lambda \neq 0$ depends on the given null orbit. We prove
 \begin{theorem}\label{Condition sur les lamdas, introduction}
 	Let $(T,K)$ be a Lorentzian torus without conjugate points, with Killing vector field~$K$. Assume that the null orbits of $K$ are incomplete. If $\lambda_1$ and $\lambda_2$ are the parameters related to any two consecutive null orbits of $K$ in the torus, then $$\lambda_1 = \lambda_2.$$
 \end{theorem}
As an immediate corollary of this theorem, it appears that these tori are limits of Lorentzian tori admitting a Killing field, and containing conjugate points (see Corollary \ref{limite de tores APC} in this paper).
\paragraph{Strategy of the proof of Theorem \ref{famille SPC, introduction}:}
Let us recall that no null geodesic in a Lorentzian surface has conjugate points (see \cite[page 291]{6} for instance), so we restrict our attention to non-null geodesics.  Let $\gamma$ be a geodesic of an $n$-dimensional Lorentzian manifold $M$. A Jacobi field is a vector field along $\gamma$ satisfying a differential equation called the Jacobi equation. There are many equivalent definitions of conjugate points; the one we will be using in this paper is the following: a pair of conjugate points on $\gamma$ are points such that there exists a non-trivial Jacobi field along $\gamma$ vanishing at these points. When $\gamma$ is not lightlike, we can suppose that this vector field is orthogonal to $\gamma$; and when $M$ is a surface, this reduces to the differential equation in one variable
\begin{align}\label{Equation de Jacobi, introduction}
u^{''}+\epsilon \kappa u=0,
\end{align}
where $\kappa$ is the sectional curvature along $\gamma$, and $\epsilon=\pm 1$ is the sign of $\langle \dot{\gamma},\dot{\gamma} \rangle$. \\

A Clifton-Pohl torus, denoted by $T_{CP}$, is the quotient of the manifold $\R^2-\{0\}$ equipped with the metric
\begin{align}\label{Tore de Clifton Pohl}
g_{CP}=\frac{2dxdy}{x^2+y^2}\;\;\;\; \,(x,y)\in \R^2-\{0\},
\end{align}
by some non trivial homothety.
The proof that $T_{CP}$ has no conjugate points is done using a remarkable property: the universal cover of $T_{CP}$ is a proper open subset of an extension $\hat{\Sigma}$, introduced in \cite{9}.  This extension is geodesically complete; thus, some of the solutions of  (\ref{Equation de Jacobi, introduction}) defined over $\R$ vanish more than one time on $\widetilde{\gamma}$, the extension of a geodesic $\gamma$ of $\widetilde{T}_{CP}$ to $\hat{\Sigma}$, but there are never two such zeros in the Clifton-Pohl torus. This is obtained through an explicit resolution of the Jacobi equations.  \\

The possibility of extending the universal cover of a Lorentzian torus to  a maximal Lorentzian surface is not specific to the Clifton-Pohl torus.  Let $(T,K)$ be a Lorentzian torus with a non-trivial Killing field $K$,  and let $(\widetilde{T},\widetilde{K})$ be its universal cover.  There exists a maximal Lorentzian surface $(E,K_E)$ homeomorphic to $\R^2$ such that $(\widetilde{T},\widetilde{K})$ is isometrically embedded in $E$ (\cite[Theorem 3.25]{10}). 
This extension is unique when some further hypotheses are added  on it; we do not recall that here. We prove that 
\begin{theorem}\label{Complétude, introduction}
	The extension $E$ associated to a Lorentzian torus $(T,K)$ is geodesically complete. 
\end{theorem}
This makes these surfaces simple to deal with. When the torus is not flat, this extension always contains conjugate points (see \cite{10}, Proposition 5.28). So one has to prove that there are never two such points in the universal cover of the torus for the family given in Theorem \ref{famille SPC, introduction} above. In this paper, the existence of conjugate points is studied from the point of view of the oscillation theory of the Jacobi equation.  One of the solutions of (\ref{Equation de Jacobi, introduction}), denoted by $\beta$ in \cite{9}, is given by the normal component of $K$ on $\gamma$.   

When condition iv) is added, we develop a rather simple point of view from which one can conclude to the existence or not of conjugate points. In fact, given a non-null geodesic $\widetilde{\gamma}$ of the extension $E$, when $\beta$ vanishes twice on $\widetilde{\gamma}$ (these are the only geodesics that could carry conjugate points in this case), the geodesic is preserved by a translation $T=4\omega$ of the geodesic parameter, and the distance  between two consecutive zeros of $\beta$ is constant, equal to $T/2$. Furthermore, when the foliation orthogonal to $K$ contains only Reeb components (this is a necessary condition to the absence of conjugate points; Theorem 5.29, \cite{10}), one proves easily that $\widetilde{\gamma}$ is contained in the torus on an interval of type $]t_0,t_0+2\omega[$. So it becomes clear that the absence of conjugate points is equivalent to the fact that $\beta$ realises the minimum distance between two consecutive zeros of the solutions of the Jacobi equation. This is the point of view we use to prove Theorem \ref{famille SPC, introduction} above. The distance between these zeros is studied by use of techniques from differential equation theory, provided in \cite{1},\cite{3} and \cite{4}. The idea used is the fact that there exists a close connection between the oscillation problem of the equation (\ref{Equation de Jacobi, introduction}) and the eigenvalue problem $$u^{''} + \lambda \epsilon \kappa u =0,$$
with suitable boundary conditions. This proof is an example of the interplay between geometry, analysis and the theory of ordinary differential equations in the study of conjugate points.\\

In the Riemannian setting, it is known from Hadamard's theorem that non-positive curvature implies that there are no conjugate points; therefore, it is easy to get open sets of Riemannian metrics without conjugate points in the $C^2$ topology. In the Lorentzian general setting, the stability question is a little harder; in fact, the restriction on the sign of the sectional curvatures is no help since the Jacobi equation involves also the type $\epsilon$ of the geodesic. If we drop the assumption concerning the additional symmetries on the universal cover of the torus, we show that the absence of conjugate points for a Lorentzian torus admitting a Killing field can be expressed in terms of the positivity of some numeric function defined on an open subset of the torus. The method used has the disadvantage of not providing a geometric ingredient to construct metrics without conjugate points, but it suggests that one can expect a stability result among the metrics admitting a Killing field. 
Denote by $\mathcal{L}_K(T)$ the space of smooth Lorentzian metrics on the $2$-torus $T$ admitting a non-trivial Killing field.  We obtain
\begin{theorem}
A metric in $\mathcal{L}_K(T)$ close enough to the Clifton-Pohl metric, for the $C^{\infty}$ topology, and satisfying the condition in Theorem \ref{Condition sur les lamdas, introduction} above, has no conjugate points.
\end{theorem}
Actually, we prove a more general stability by deformation result in the last section of this paper. This stability result ensures the existence of Lorentzian metrics without conjugate points and admitting a Killing field, without the symmetries added in Theorem \ref{famille SPC, introduction}. Of course, an important hypothesis all along this paper is the existence of a Killing field. The existence of a metric without conjugate points which does not satisfy this additional assumption is an open question.
\\ 
 
The paper is organized as follows: in paragraph 2 we introduce the fundamental tools and notions from \cite{10} dealing with the classification of compact Lorentzian surfaces with a Killing field, and prove the geodesic completeness of the maximal extensions associated to such surfaces. Paragraph 3 is a study of the Jacobi equation regardless of geometry; we establish some lemmas about the distance between the zeros of the solutions of such an equation. These lemmas will be applied in paragraph 4 in the case of Lorentzian tori with a Killing field, in which we characterize Lorentzian tori without conjugate points in the way presented before. The obstruction in Theorem \ref{Condition sur les lamdas, introduction} will follow from this characterization. The last paragraph studies the stability character of the property of being without conjugate points, and prove that some of the examples obtained are as stable as possible in $\mathcal{L}_K(T)$.

\begin{center}
	Acknowledgements
\end{center}

The author would like to warmly thank her thesis advisor Christophe Bavard for his support and encouragements, and the long hours of discussion he devoted to her, without which this article would not have been completed.
\section{Universal extensions of Lorentzian tori with a Killing field}
\subsection{Structure of Lorentzian tori with a Killing field}

All the facts we will be compiling in this section, dealing with the geometry of Lorentzian tori admitting a Killing field, have been investigated in \cite{10}, in the more general setting of Lorentzian surfaces with a Killing field.  For convenience, we set up in the first place some non-trivial vocabulary and notations, and then give the relevant results from \cite{10} we will be using in this paper without proofs. Let $(X,K)$ be a Lorentzian surface with a Killing field $K$, which we assume to be complete. 
\begin{definition}(ribbons, bands and dominoes) \label{Définition ruban/domino}
Let $U$ be a subset of $X$ saturated by $K$. Suppose $K$ never vanishes on $U$; we say that $(U,K)$ is \\
(1) a ribbon if $U$ is open, simply connected  and if one of the null foliations in $U$ is everywhere transverse to $K$. \\
(2) a band if $U$ is homeomorphic to $[0,1] \times \R$, with $\langle K,K \rangle$ vanishing on the boundary and not vanishing in the interior of $U$. \\
(3) a domino if $U$ is open, simply connected, and $K$ has a unique null orbit in $U$.
\end{definition}
If $U$ denotes a ribbon in $X$ and $p \in U$, we can choose a null-geodesic $\gamma$ passing through $p$, maximal in $U$ and transverse to $K$.  On the saturation of the geodesic by the flow of $K$, which is equal to $U$ by connexity, the metric writes   $$2dxdy + f(x)dy^2,$$ where $L=\partial_x$ is a null vector field parametrized such that $\langle L , K \rangle =1$ (Clairaut's lemma shows that this quantity is indeed constant), and $K=\partial_y$.
The coordinate denoted by $x$, which is well defined up to translation and change of sign, will be called the "transverse coordinate", or simply the $x$-coordinate. Thus, the norm of $K$ is given by $f$ in the $x$-coordinate; it vanishes on the null orbits of $K$ contained in $U$ and transverse to $L$.  \\
When $\langle K , K \rangle (p) \neq 0$, there exists another null-geodesic transverse to $K$ and passing through $p$, giving rise to another formula for the metric on an open set $U'$ of $X$. On the intersection $U \cap U'$, the norm of $K$ doesn't vanish: we have by  Proposition 2.5 of  \cite{10} the existence of a generic reflection, i.e. a local isometry fixing a non-degenerate geodesic perpendicular to $K$ and sending $K$ to $-K$, and thus permuting the null foliations. The transition map is given on $U \cap U'$ by composing $(x,y) \mapsto (-x,-y)$ with the generic reflection $\phi$ given by $\phi(x,y)= (-x,2 G(-x) + y)$, where  $G$ is a primitive function of  $-1/f$.  
\begin{example}\textbf{An atlas for the Clifton-Pohl torus, $T_{CP}$}\\ 
Choosing the right parametrization of the null foliations, i.e. such that $\langle L , K \rangle = \langle L' , K \rangle=1$, we show that an atlas for the Clifton-Pohl torus is given by open sets $U_i=I_i \times \R$ $(i=0,1,2,3)$, with $I_i=]\frac{i\pi}{2},\frac{i\pi}{2}+\pi[$, each of them equipped with the metric $2dudv+sin(2u)dv^2$. A generic reflection on $U_i \cap U_{i+1}$ is given by $\phi(u,v)= (u,log|\frac{sin(2u)}{1-cos(2u)}|-v)$; it sends $K$ to $-K$ and $L$ to $-L'$. On two local charts $U_i$ and $U_{i+1}$ where the metric is given by $2dudv+sin(2u)dv^2$ and $2du'dv'+sin(-2u')dv'^2$ respectively, one gets a change of coordinates by composing $(u,v) \mapsto (-u,-v)$ with a generic reflection. 	
\end{example} 
The connected components of $X-\{ \langle K , K \rangle =0\}$ are bands (in fact, interior of bands). A band is said to be spacelike (resp. timelike) if $K$ is spacelike (resp. timelike) in this band. We have the following definition, resulting from Lemma 2.8 of \cite{10}:
\begin{definition}
A Lorentzian band $(B,K)$ is said to be:\\
(1) of type I: if the foliations defined by $K$ and $K^{\perp}$ are both suspensions. \\
(2) of type II: if the foliation of $K$ is a suspension and that of $K^{\perp}$  is a Reeb component. \\
(3) of type III: if the foliation of  $K$ is a Reeb component and that of $K^{\perp}$ is a suspension.
\end{definition} 
In the following figure, the foliation of $K$ is represented by continuous lines, the orthogonal foliation by dotted lines.
\begin{figure}[h!] 
	\labellist 
	\small\hair 2pt 
	\pinlabel {$\text{Type I band}$} at 125 -29 
	\pinlabel {$\text{Type II band}$} at 532 -29
	\pinlabel {$\text{Type III band}$} at 935 -29
	\endlabellist 
	\centering 
	\includegraphics[scale=0.30]{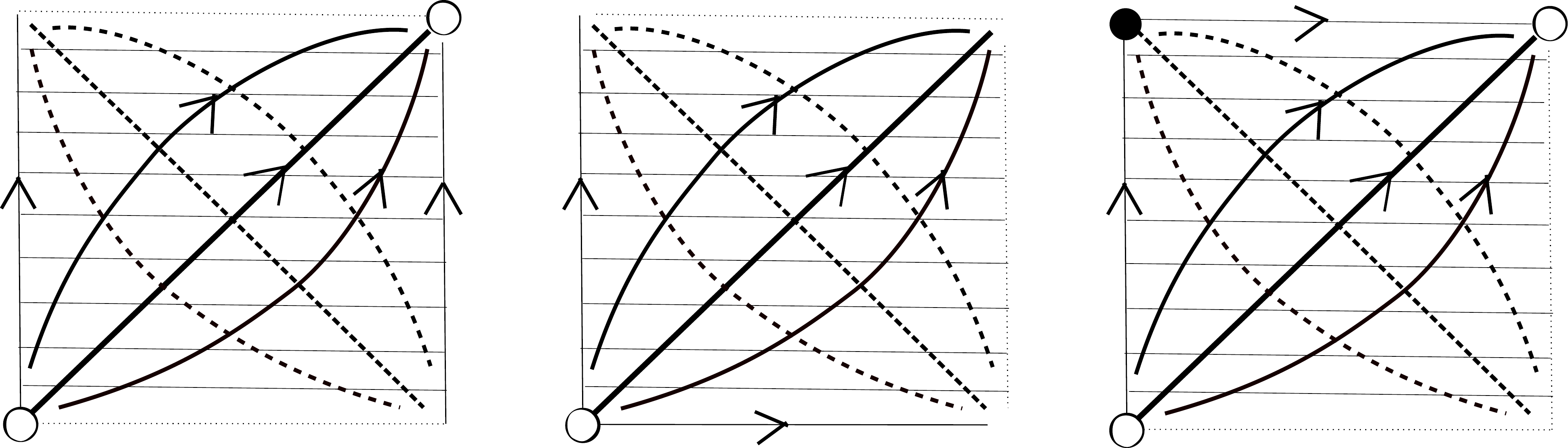} \caption{A type I, type II and type III band} 
	\label{fig:cobo}
\end{figure} 
\newpage
The Clifton-Pohl torus contains 4 maximal bands of type II. \\

Let $(T,K)$ be a non-flat torus with a Killing vector field $K$. The fact that the integral curves of $K$ are  closed (see the proof of Theorem 3.25, \cite{10}) ensures the existence of a closed curve everywhere transverse to $K$. It follows that $T$ does not contain type III bands. Besides, $T$ contains a finite number of type II bands. Now, consider a curve $\gamma$ in $\widetilde{T}$ made of broken null geodesics, everywhere transverse to $\widetilde{K}$, such that the bifurcation points are contained in the type II bands of the torus. This curve is parameterized by $| \langle \dot{\gamma},\widetilde{K} \rangle |=1$. Take a point $p$ on $\gamma$ and set $\langle \dot{\gamma}(p),\widetilde{K}(p) \rangle =1$; this defines an orientation on $\gamma$.  We notice that $\langle \dot{\gamma},\widetilde{K} \rangle$ changes sign at each bifurcation point.  
Denote by $\gamma_i$ the maximal geodesics contained in $\gamma$ with the induced orientation; each $\gamma_i$ is contained in a maximal ribbon in which we have coordinates $(x,y)$ constructed as above, such that $K=\partial_y, \langle \partial_x,\partial_x \rangle =0$, and $\langle \partial_x,\partial_y \rangle =(-1)^i$. These ribbons cover the torus. Define locally in each ribbon $\nu = dy \wedge dx$; it is easily seen that $\nu$ is a well defined volume form on $\widetilde{T}$. This gives rise to a submersion $\textbf{x}:\widetilde{T} \to \R$  defined (up to translation) by $i_K \nu = d\textbf{x}$, inducing a global diffeomorphism between the space of the leaves of $\widetilde{K}$, denoted by $\mathcal{E}_{(\R^2,\widetilde{K})}$ in \cite{10}, and $\R$, thus making it into a Haussdorf manifold of dimension $1$ (see Proposition 2.21 \cite{10}). The norm of $\widetilde{K}$ is then factorized into $\langle \widetilde{K} ,\widetilde{K} \rangle= f o \textbf{x}$, where $f$ is defined on $\mathcal{E}_{(\R^2,\widetilde{K})}$, identified with $\R$. The submersion $\textbf{x}$ coincides, up to translation and change of sign, with the $x$-coordinate of any local chart. The $f(x)$ function so obtained is periodic. \\

Now let $(p_{k})$ denote the sequence of zeros of $f$ in the $\textbf{x}$ coordinate, taken in an increasing order, such that the ribbons defined by $U_k = (I_k, (-1)^k 2dxdy + f(x) dy^2)$, where $I_k=]p_{k},p_{k+1}[$, are maximal in the torus (this set of zeros is in fact discrete, for type II bands don't accumulate). 
We glue the open sets $U_k$ and $U_{k+1}$ by means of the local isometries $\psi_k(x,y) = (x, 2 G(x) +y)$, where $G$ is a primitive function of $-1/f$ on $I_k \cap I_{k+1}$, thus obtaining the universal cover of the torus equipped with its Killing vector field  $\widetilde{K}$. This construction gives an atlas for the universal cover; we say that we have a structure modeled on $E^u_f$, or $E^u_f$-structure, on $(\widetilde{T},\widetilde{K})$ which, locally, is only determined by the function $f$. Such a structure exists on a connected and saturated Lorentzian surface each time the norm function of the Killing vector field factorizes in such a way (see Proposition 3.19, \cite{10}). \\

Conversely, if we are given a periodic function $f$, we can define in the same way a Lorentzian surface which is the universal cover of a torus. Observe that if the $I_i$'s are chosen so as each of them contains only one zero of $f$, we get a surface all of whose bands are of type II, for two consecutive zeros of $\langle K,K \rangle$ belong to different null foliations. If we choose the $I_i$ with an arbitrary number of zeros of $f$ in it, the torus shall contain type I bands also.  
\begin{definition}(Minimal number of bands)\label{Definition n=minimal number of bands}
We assume the connected components of $\{f \neq 0\}$ don't accumulate. Define $\mathfrak{n}$ to be the number of zeros of $f$ on a period. This number corresponds to the number of bands of the "smallest" quotient of the universal cover giving a torus with the same $E^u_f$ structure.  
\end{definition}  
Thus, for the Clifton-Pohl torus we have $\mathfrak{n}=2$, for it is itself a two-sheeted covering of a torus containing $2$ bands and having the same $E^u_f$ structure as $T_{CP}$. \\
We won't consider the case in which these connected components accumulate since the locally finite hypothesis is a neccessary condition for a torus admitting a Killing vector field to be without conjugate points (see Theorem 5.29, \cite{10}). \\
\\
The possibility of extending the universal cover of a Lorentzian torus into a Lorentzian surface all of whose null geodesics are complete (we say $L$-complete) will be used in an essential way in this paper; we call the Lorentzian surface so obtained the "universal extension" of the torus. We shall recall some essential properties of this surface. Let $I$ be a non-empty open interval of $\R$  and let $f: I \to \R$ be a smooth function. Denote by $R_f=(R,K)$ the surface $(R=I \times \R$, $2dxdy + f(x) dy^2)$, $(x,y) \in R$, with a Killing vector field $K= \partial_y$; it is called the "ribbon associated to $f$".  $R_{f^{-}}$ refers to the ribbon associated to $f^{-}$, defined for $-x \in I$ by $f^{-}(x)=f(-x)$. 
\begin{theorem}(Proposition 3.3, Theorem 3.25 \cite{10}).\label{Existence d'une extension maximale}
Let $(T,K)$ be a Lorentzian torus with a non-trivial Killing vector field $K$, $(\widetilde{T},\widetilde{K})$ its universal cover. Let $f: \R \to \R$ denote the function induced by $\langle \widetilde{K} , \widetilde{K} \rangle$  in the way set before. There exists a maximal Lorentzian surface $(E^u_f,K^u)$ homeomorphic to $\R^2$ such that $(\widetilde{T},\widetilde{K})$ is isometrically embedded in $E^u_f$, and $E^u_f$ is reflexive and $L$-complete. This extension is unique and every maximal ribbon contained in $(E^u_f,K^u)$ is isometric to either $R_f$ or $R_f^{-}$.
\end{theorem} 
The Lorentzian surface $E^u_f$ satisfies, in addition, the following remarkable reflexivity property: 
\begin{proposition}(Lemma 3.10 \cite{10})\label{Extension des réflexions}
The generic reflections in any ribbon contained in $E^u_f$ extend to global isometries of $E^u_f$.
\end{proposition} 
Let $(T,K)$ be a Lorentzian torus with Killing field $K$, and let $E^u_f$ be the extension of the universal cover, given in Theorem \ref{Existence d'une extension maximale} above. This Lorentzian surface is obtained using two operations. Roughly speaking, they consist in: \\
1. gluing copies of $R_f$, the maximal ribbon defined above, along the bands, using generic reflections. This operation extends the null-geodesics interior to the bands of the ribbons into complete geodesics;\\
2. adding saddles: a saddle is obtained in \cite{10}, Proposition 2.29, as the extension of a domino (whose unique null orbit of $K$ is incomplete) by a simply connected surface containing a unique zero of $K$. This extends the null orbits of $\widetilde{K}$ which are geodesically incomplete into complete geodesics. With a good choice of the generic reflections in the first operation, one can make this extension compatible with the surface we get in 1 (see the proof of Proposition 3.3, \cite{10}, for details). Denote by $U$ such a domino, and by $\widetilde{U}$ the extension. Write $U=I \times \R,\, 2dxdy + f(x) dy^2$ in the local coordinates, with $f(0)=0$ and $f^{'}(0)=\lambda$, $\lambda \neq 0$ (Remark \ref{Orbites incomplètes=zéros simpes}). The metric on $\widetilde{U}$ reads (Proposition 2.29, \cite{10})
\begin{align}\label{Selle: expression explicite de la métrique}
\frac{1}{\lambda} [v^2 h(uv) du^2 -2 (j(uv)+\frac{1}{j(uv)}) dudv + u^2 h(uv) dv^2]; \;\;\;uv \in I, v \in \R,
\end{align} 
where $x=uv; j,h \in C^{\infty}(I,\R)$, such that $j(x)= \int_{0}^{1} f^{'}(tx) dt$ and $h(x)=\int_{0}^{1} l^{'}(tx)dt,$ \,with $l(x)=j(x)-\frac{1}{j(x)}$, and the Killing field corresponds to $K = \frac{2}{\lambda}(u \partial_u - v \partial_v)$.
\\

We want to know how these objects, i.e. $R_f$ and the saddles, depend on the metric. Actually, we will see that if the Killing field depends smoothly on the metric, then these objects do, too.\\
Let $\widetilde{T}$ be a universal cover of $T$, and $\widetilde{K}$ the Killing field on it. Recall that the norm of $\widetilde{K}$ can be written  $\langle \widetilde{K}, \widetilde{K} \rangle =f o \textbf{x}$, where $\textbf{x} \in C^{\infty}(\widetilde{T},\R)$ is defined up to translation and change of sign, and $f \in C^{\infty}(\R,\R)$. We can make the construction of $\textbf{x}$ more geometric by looking at the proof of Proposition 2.21, \cite{10}. Fix $p \in \widetilde{T}$, and consider a positively oriented hyperbolic basis $(X,Y)$ in $\mathfrak{X}(\widetilde{T})$, i.e. $\langle X,X \rangle=\langle Y,Y \rangle=0$ and $\langle X,Y \rangle=1$. Define a volume form $\nu$ by setting $\nu(X,Y)=1$; $\nu$ does not depend on the choice of this basis. Define a 1-form $\omega :=i_{\widetilde{K}} \nu$; this form is closed, hence exact since $\widetilde{T}$ is simply connected, so there exists (a unique) $\textbf{x} \in C^{\infty}(\widetilde{T},\R)$ such that $\omega=d\textbf{x}$ and $\textbf{x}(p)=0$. It is easy to check that this function $\textbf{x}$ coincides with the one defined before. \\
The space $\mathfrak{X}(T)$ of smooth vector fields on $T$, together with $\mathcal{L}_K(T)$, are equipped with the $C^{r}$ topology. 
\begin{lemma}\label{Convergence des rubans}
Let $(g,K) \in \mathcal{L}_K(T) \times \mathfrak{X}(T)$ a non-flat metric on $T$, such that $\mathcal{L}_K g =0$, and let $(g_n,K_n) $ be a sequence in $ \mathcal{L}_K(T) \times \mathfrak{X}(T)$ such that $\forall n, \mathcal{L}_{K_n} g_n =0$ and $(g_n,K_n) \overset{C^{r}}{\to} (g,K)$. Then the sequence of ribbons $R_{f_n}$, where $f_n$ denotes the function induced by the norm of $K_n$, converges to $R_f$, where $f$ is induced by the norm of $K$, for the $C^r$ topology. 
\end{lemma}
\begin{proof}
For all $n$, one can choose a Lorentzian volume form $\nu_n$ that defines the transverse coordinate $x_n$ by setting $i_{\widetilde{K}_n}  \nu_n = d  \textbf{x}_n$ and $\textbf{x}_n(p)=0$, such that the sequence 
$\nu_n$ converges to $\nu$. By definition $\textbf{x}_n$  converges $C^r$ to $\textbf{x}$ on every compact subset of $\widetilde{T}$. Now, take a curve $c$ everywhere transverse to $\widetilde{K}$ and $\widetilde{K}_n$ for $n$ large enough, cutting each leaf of $\widetilde{K}$ only one time. Define $I=\{c(t), t\in \R \}$ on which  $\textbf{x}$ and $\textbf{x}_n$ are diffeomorphisms. Writing $f=g(\widetilde{K}, \widetilde{K}) o \textbf{x}^{-1}$, we get the $C^r$ convergence of $f_n$ to $f$ on every compact subset of $\R$, hence everywhere, since they are periodic. This proves the lemma.  
\end{proof}
\begin{lemma}\label{Convergence des selles}
Let $(U,K)$, $U=I \times \R$, be a Lorentzian domino where the unique null orbit of $K$ (represented by $x=0$) is incomplete. Denote by $g$ the metric on $U$ and let $(g_n,K_n)$ be a sequence of metrics on $U$ such that $(g_n,K_n) \overset{C^{r}}{\to} (g,K)$, with $r \geq 2$. Then, there exists a neighborhood $J$ of $0$ such that the extension  of $V := J \times \R$ for the metric $g_n$, denoted by $\widetilde{V}_n$, converges $C^{r-2}$ to $\widetilde{V}$, the extension for the metric $g$.
\end{lemma}
\begin{proof}
Before starting the proof, let us state the following fact: 
\begin{fact}\label{Fait sur la convergence des zéros uniques de fonctions sur un compact}
let $F_n$ be a sequence of functions defined on a compact manifold $M$ with values in $\R$, which converges uniformly to a function $F$. Suppose that $ F $ admits a unique zero in $ M $ -denote it by $ p $, and that $ F_n $ admits a unique zero $p_n$ in $ M $ for all $ n \in \N $. Then the sequence $ p_n $ converges to $ p $ in $ M $.
\end{fact}
Now, let $p$ be a point on the null orbit of $K$ in $U$; the transverse coordinate for the metric $g$ satisfies $x(p)=0$. On a neighborhood of this orbit, there is a unique null orbit of $K_n$ for all $n$ but a finite number (the zero of $f$ in $U$ is simple by Remark \ref{Orbites incomplètes=zéros simpes}), so let $p_n$ be a point on it such that $p_n$ converges to $p$ (this is possible by Fact \ref{Fait sur la convergence des zéros uniques de fonctions sur un compact} above). Denote by $x_n \in C^{\infty}(U,I_n)$ the transverse coordinate of $g_n$ such that $x_n(p_n)=0$. Since $I_n$ converges to $I$ and $f_n$ is $C^r$ close to $f$, we can find a neighborhood  $J$ of $0$ on which $f$ and $f_n$ are all defined for $n$ big enough, with $f_n$ having only one zero on $J$. The conclusion of the lemma follows then from Equation (\ref{Selle: expression explicite de la métrique}) that gives the explicit expressions of the extensions. 
\end{proof}
\begin{remark}
We point out that these lemmas state the convergence of the ribbons and the saddles as abstract objects related to the metric, i.e. depending only on $f$. We didn't say anything about the local coordinates $(x,y)$ and $(u,v)$. Actually, we will see in Paragraph \ref{Section: convergence des champs de Killing} that these local charts depend smoothly on the metric.
\end{remark}
\subsection{Completeness of the universal extensions of non-elementary tori}
We know from \cite{9} that the extension associated to the Clifton-Pohl torus is complete. In this section, we prove that this is true for any extension $E^u_f$, with $f$ periodic. Null-completeness being already obtained, what we have to show is that non-null geodesics are complete.   
\begin{theorem}\label{Complétude}
The extension $E^u_f$ associated to a Lorentzian torus $(T,K)$ is geodesically complete. 
\end{theorem}
A maximal geodesic $\gamma$ may have two different behaviors: either it leaves any maximal ribbon contained in $E^u_f$, or it remains in a maximal ribbon provided $t$ goes close enough to the limit. In the second case, we shall consider two behaviors again: set $I=\{x(\gamma(t))\}$; as $t$ ranges over the domain of $\gamma$, either $I$ is bounded, in which case the geodesic remains in a band as $t$ approaches the limit of the domain of $\gamma$, or $I$ is unbounded.\\
Denote by $T$ the unit vector field tangent to $\gamma$, and $N$ the vector field along $\gamma$ orthogonal to $T$, such that the basis $(T,N)$ is positively oriented. Set $$K=C T + \beta N.$$
Then $C=\epsilon \langle T,K \rangle$ is a constant called the Clairaut constant (see \cite{9}, p. 3), and $\beta=-\epsilon \langle K,N \rangle$ is a solution of the Jacobi equation.  Notice that if $K(p), p \in \gamma$, is not degenerate, $\beta(p)=0$ if and only if $\gamma$ is tangent to $K$ at $p$. 
\begin{lemma}\label{asymptote à K}
	Let $\gamma$ be a non-null geodesic such that $\beta$ vanishes at most one time. If $\gamma$ remains in a band for $t$ large enough, it asymptotically approaches a leaf of $K^u$; if $\gamma$ is not perpendicular to $K^u$, this leaf is either timelike or spacelike, depending on the type of $\gamma$.
\end{lemma}
\begin{proof}
	 Provided we go far enough out in $\gamma$, we may suppose that the geodesic is transverse to $K^u$ in the band. In the coordinates $(x,y)$, this amounts to saying that the derivative $x^{'}$ does not vanish on $\gamma$; the $x$-coordinate  is therefore strictly monotone and converges, since it is bounded, to a constant $x_0$. Moreover, the $y$-coordinate is strictly monotone on $\gamma$. If $y$  converges to $y_0$ on $\gamma$, it is easy to see that the geodesic can be extended beyond the point $p=(x_0,y_0)$. Indeed, following the proof of Lemma 8 p.130, \cite{6}, take a convex neighborhood $\mathcal{V}$ of $p$ (an open set is said to be convex provided it is a normal neighborhood of each of its points). The geodesic $\gamma$ is contained in $\mathcal{V}$ for $t \geq a$, for some $a>0$; set $q=\gamma(a)$. In particular, there is a unique geodesic segment $\alpha:[0,1] \to T$ joining $p$ and $q$, that lies entirely in $\mathcal{V}$. This geodesic coincides with $\gamma$ and extends it past $p$. \\
	It follows that the $y$-coordinate goes to infinity on $\gamma$ while the geodesic approaches the leaf of $K^u$ corresponding to $x_0$. Now, we want to prove that $x^{'}$ tends to $0$. We start with the following fact:
	\begin{fact}
		If the leaf of $K^u$ corresponding to $x_0$ is not a null orbit of $K^u$, i.e.$f(x_0) \neq 0$, then $x^{'}$ tends to $0$ on $\gamma$; in particular, $C^2 = \epsilon f(x_0)$.
	\end{fact}
		Indeed, assume that $f(x_0) \neq 0$; it appears from equations (\ref{Intégrales premières,2}) and (\ref{x' borné, juste ici}) below that both $x^{'}(t)$ and $y^{'}(t)$ converge as $x$ goes to $x_0$ on $\gamma$, with a finite limit. Now regard $x$ as a function of $y$ and write  $\frac{\mathrm{d}x}{\mathrm{d}y}=\frac{x^{'}(t)}{y^{'}(t)}$; this derivative converges on $\gamma$; denote its limit by $l \in \bar{\R}$. Of course this limit cannot be infinite; indeed, $x$ regarded as a function of $y$ is strictly monotone and tends to a constant while $y$ goes to infinity; so if $\frac{\mathrm{d}x}{\mathrm{d}y}$ converges to $l$, this limit is necessarily zero. The conclusion that $x^{'}(t)$ tends to zero is straightforward, and shows, using equation (\ref{x' borné, juste ici}), that the leaf of $K^u$ corresponding to $x_0$ has the same type as $\gamma$.
	
	\textbf{Case where $C \neq 0$:} 
	Without loss of generality we can assume  that  $\gamma$ is spacelike. We will show that  $f(x_0) \neq 0$, which will end the proof. We may suppose that $x^{'} >0$ on $\gamma$, by changing $K^u$ to $-K^u$ in the local chart, if necessary. Suppose, contrary to our claim, that $f(x_0)=0$,  and call $\gamma_{\infty}$ the corresponding light orbit of $K^u$. In the band containing $\gamma$ (for $t$ close enough to the limit), $\gamma_{\infty}$ and its image $\widetilde{\gamma}_{\infty}$ by a generic reflection, the space of the leaves of $K^u$ is a simple branched line (thus a non-Hausdorff space), in which the branched points correspond to the null orbits of $K^u$. The claim that $x$ goes to $x_0$ on $\gamma$, with $f(x_0)=0$, means that $\gamma$ approaches one of the two branched points, i.e. either $\gamma_{\infty}$ or $\widetilde{\gamma}_{\infty}$.   Let us state the following plain observation  
	
	\textbf{Observation:}\label{Observation: combien de vecteurs T satisfont  <T,K> =C}
		The band containing $\gamma$ is contained in two maximal ribbons; call $R_h$ the maximal ribbon containing $\gamma_{\infty}$ and $R_v$  the one containing $\widetilde{\gamma}_{\infty}$. Denote again by $(x,y)$ the coordinates on $R_h$, and $(u,v)$ the coordinates on $R_v$ such that $K^u=\partial_v, L^{'}= \partial_u,$ and $ \langle ~\partial_u,\partial_v \rangle = -1$. Let $p$ be a point in $\bar{R}_h \cap \bar{R}_v$; we have two cases: \\
		1) $\langle K^u(p),K^u(p) \rangle =0$: in this case, there is a unique $T \in T_p(T)$ such that $\langle K^u(p),T \rangle =~C$ and $\langle T,T \rangle =1$, for $C \neq 0$, and this vector is defined by $T=(C,\frac{1}{2C})$ if the null orbit of $K^u$ containing $p$ is in $R_h$,  and by $T=(-C,\frac{1}{2C})$ in the $(u,v)$-coordinates, if the orbit is in $R_v$.   \\
		2) $\langle K^u(p),K^u(p) \rangle \neq 0$: there are two vectors $T_1, T_2 \in T_p(T)$ such that $\langle K^u(p),T \rangle = C$ and $\langle T,T \rangle =1$; in the $(x,y) $ coordinates, they are given by 
		\begin{align}
	    T_1=&(\sqrt{C^2 -f(p)}, \frac{C-\sqrt{C^2 -f(p)}}{f(p)}),\\	T_2=&(-\sqrt{C^2 -f(p)},\frac{C+\sqrt{C^2 -f(p)}}{f(p)}),
		\end{align}
		and we have
	\begin{center}
			\begin{tabular}{|l|l|l|l|l|l|}
			\hline
			$C$ & $f(p)$ & $x_1^{'}$ & $y_1^{'}$ & $x_2^{'}$ & $y_2 ^{'}$\\
			\hline
			$+$ & $+$ & $+$ & $+$ & $-$ & $+$  \\
			\hline
			$+$ & $-$ & $+$ & $+$ & $-$ & $-$ \\
			\hline
			$-$ & $+$ & $+$ & $-$ & $-$ & $-$ \\
			\hline
			$-$ & $-$ & $+$ & $+$ & $-$ & $-$ \\
			\hline
		\end{tabular}
	\end{center}
	where $(x^{'}_i,y^{'}_i)$ are the coordinates of $T_i,\;\, i=1,2$. \\
	
	 We distinguish two different behaviors of $\gamma$ according to the sign of $C$. First, assume that $C > 0$; in this case,  $x^{'}(t),y^{'}(t) >0$. Let $q$ be a point on $\gamma_{\infty}$, and denote by $\alpha$  the geodesic in $R_h$ tangent to $T \in T_q(T)$, such that $\langle T,K^u(q) \rangle =C$ and $\langle T,T \rangle =1$; it appears from the first part of the observation above that the two coordinates of $\alpha^{'}(s)$ at $q$ are positive, hence remain positive all along the geodesic. Now take $\tau \in \R$ such that for $t \geq \tau$, $\gamma$ is transverse to $K^u$ (the existence of $\tau$ is guaranteed by the assumption that $\beta$ does not vanish for $t$ large enough); denote by $x_{\tau}$ the coordinate of the orbit of $K^u$ intersecting $\gamma$ at $\gamma(\tau)$. From $C^2 = \beta^2 +f$, one gets $C^2 \geq \displaystyle \sup_{t \geq \tau} f(\gamma(t))$, hence $C^2 > \displaystyle \sup_{[x_{\tau},x_0]} f(x)$. It follows that $\alpha$, whose Clairaut's constant is also $C$, is defined and transverse to $K^u$ on the segment $[x_{\tau},x_0]$, using the previous fact. So by moving $q$ on $\gamma_{\infty}$ by the flow of $K^u$ if necessary, one can suppose that $\gamma$ and $\alpha$ intersect at a point $p$ of the orbit of $K^u$ having coordinate $x_{\tau}$.  At this point, we have $\langle \gamma^{'}(p),K^u(p) \rangle = \langle \alpha^{'}(p),K^u(p) \rangle = C$, and the coordinates of $\alpha^{'}(p)$ are both positive. Now using the observation above, we see that when $C>0$, only one of the two vectors satisfying $\langle T,K \rangle =C$ and $\langle T,T \rangle~=~1$, has the additional property that both of its coordinates are positive.   This clearly forces  $\alpha^{'}(p)= \gamma^{'}(p)$, hence $\gamma$ extends beyond the band by cutting $\gamma_{\infty}$ transversally. Of course this contradicts our assumption, and proves that actually, $f(x_0) >0$.  Now, to deal with the case $C <0$, write $\gamma^{'}(t) =(u^{'}(t),v^{'}(t))$ in the $(u,v)$ coordinates, with 
	 \begin{align}
	 u^{'}&=x^{'},\\
	 v^{'}&=\frac{1}{f(x)}(f(x) y^{'} +2x^{'}).
	 \end{align}
	 Write $\langle \gamma^{'}(t),\gamma^{'}(t) \rangle =y^{'}(2x^{'}+fy^{'})=1$. When $f<0$ (resp. $f>0$), we have $x^{'} >0$ and $y^{'}>0$ (resp. $y^{'} <0$), hence $2x^{'}+fy^{'} >0$ (resp. $2x^{'}+fy^{'} <0$). This gives $u^{'} >0$ and $v^{'} <0$ in both cases. Let $\widetilde{q}$ be a point on $\widetilde{\gamma}_{\infty}$, and $\widetilde{\alpha}$ the geodesic in $R_v$ tangent to $\widetilde{T} \in T_{\widetilde{q}}(T)$, such that $\langle \widetilde{T},K^u(\widetilde{q}) \rangle =C$ and $\langle \widetilde{T},\widetilde{T} \rangle =1$. Repeating the previous argument shows that $\gamma$ can be extended using $\widetilde{\alpha}$ beyond the band, which leads to a contradiction, and finishes the first part of the proof.
	 
	 \textbf{Case where $C=0$:} In this case, equation (\ref{x' borné, juste ici}) below reads $x^{'}(t)^2 = -\epsilon f(x)$. This yields $f(x_0)=0$ using again the previous fact, hence $x^{'} \to 0$, which completes the proof.  
\end{proof}
\begin{proposition}\label{Complétude dans une bande}
	Under the same conditions on $\beta$, a non-null geodesic that lies in a band after a certain while is complete.
\end{proposition}
\begin{proof}
	According to the previous lemma, the geodesic asymptotically approaches  a leaf of $K^u$.  Write $2 dxdy + f(x) dy^2$ ($x \in I$) for the metric in local coordinates. The equations $\langle \gamma'(t) , \gamma'(t) \rangle = \epsilon$ and $\langle \gamma'(t) , K^u \rangle= \epsilon C$ in the $(x,y)$-coordinates read:
	\begin{align}\label{Intégrales premières,1}
	2 &x'(t) y'(t) + f(x) y'(t)^2 = \epsilon 
	\end{align} 
	\begin{align}\label{Intégrales premières,2}
	x'&(t) + f(x) y'(t) = \epsilon C
	\end{align}
	We get
	\begin{align}\label{y' borné, juste ici}
	f(x) y'(t)^2 - 2\epsilon C y'(t) + \epsilon =0, 	
	\end{align} 
	and \begin{align}\label{x' borné, juste ici}
x^{'}(t)^2= C^2 - \epsilon f(x). 	
	\end{align}
	This yields $$t(x_0) - t(x) = \int_{x}^{x_0} \frac{dx}{\sqrt{C^2-\epsilon f(x)}} .$$
	Suppose $C=0$; in this case, $x_0$ is not a simple zero of $f$, for if $x_0$ is a simple zero of $f$, the corresponding null orbit of $K^u$ is extended in $E^u_f$ by adding a saddle point (see \cite{10} for details), and $\gamma$ leaves the band through it. Thus the above integral goes to infinity. Now, if $C \neq 0$, $\gamma$ approaches a non-null orbit of $K^u$ so that $f(x)$ is bounded away from $0$ as $\gamma$ goes to this orbit; it follows from (\ref{y' borné, juste ici}) that $y'(t)$ is bounded. Since $y$ goes to infinity on $\gamma$, the latter is necessarily complete. 
\end{proof}
\begin{lemma} \label{Behavior of geodesics, paragraph completeness}
	Let $\gamma$ be a non-null geodesic not perpendicular to $K^u$. Assume $\gamma$ cuts a null orbit of $K^u$, denoted by $\mathfrak{l}$. Then \\  
	(i)	 $\gamma$ does not cross a type II band containing $\mathfrak{l}$;  \\
	(ii) if $\beta$ does not vanish, $\gamma$ lies in the maximal ribbon containing  $\mathfrak{l}$.
\end{lemma}
\begin{proof}
	(i) By definition, the foliation of $K^{u \perp}$ in a type II band is a Reeb component; a geodesic that crosses a type II band is tangent to a leaf of $K^{u \perp}$, hence coincides with it, since the leaves of $K^{u \perp}$ are geodesics. \\
	(ii) If $\gamma$ crosses a type III band containing $\mathfrak{l}$, $\beta$ vanishes in that band. Indeed, the foliation of $K^u$ in a type  III band is a Reeb component, so a geodesic that crosses a type III band is tangent to $K^u$. Combining this with (i) gives (ii).
\end{proof}
We now go back to the proof of Theorem \ref{Complétude}:
\begin{proof}
When $f$ does not change sign, i.e. $K$ is either timelike or spacelike, the torus is obviously complete; indeed, in this case, $|\inf \langle K,K \rangle| > 0 $, and for any  non-null geodesic $\gamma : [0,b[ \to T$, equations (\ref{Intégrales premières,2}) and (\ref{x' borné, juste ici}) imply that the closure of the image of $\gamma^{'}$ in $T(T)$ is compact. The same reasoning can actually be adapted to null geodesics, so that $E^u_f= \widetilde{T}$. More generally, if $M$ is a compact Lorentzian manifold,  with a timelike (or spacelike)   Killing vector field, then $M$ is complete, like in the Riemannian case. So now, assume that $f$ changes sign.

\textbf{1$^{st}$ case: $\beta$ vanishes at least twice on $\gamma$ }\\
Under this assumption, $\gamma$ is preserved by a translation of the parameter $t$ of the geodesic; we call it periodic (see Lemma \ref{lemme géod périodique} for a proof); this gives a complete geodesic.\\
It appears from the proof of Lemma \ref{lemme géod périodique} that either $\gamma$ remains in a band or it leaves any maximal ribbon of $E^u_f$. The first situation occurs when $\beta$ vanishes twice in the interior of a band; in this case, $\gamma$ is invariant by the action $\phi^{\tau}$ of the flow of $K^u$, for a certain $\tau \in \R$. 

\textbf{$2^{nd}$ case: $\beta$ vanishes at most once on $\gamma$}\\
In this case, $\gamma$ remains in a maximal ribbon after a certain while, by Lemma \ref{Behavior of geodesics, paragraph completeness} above. If $\beta$ vanishes, set $t=0$ at this point; if it doesn't vanish, take an arbitrary point $p$ on $\gamma$ and set $t=0$ at this point. Let $I=\{ x(\gamma(t)), t>0 \}$. Suppose that $I$ is unbounded.
Since $f$ changes sign, fix a band in which $f$ has sign $\epsilon$; in this band $$t(x_1) - t(x_0) \geq \int_{x_0}^{x_1} \frac{1}{\sqrt{C^2}} dx =a_0. $$
Now $\gamma$ crosses infinitely many such bands, and since the time $\gamma$ takes to cross a band depends only on $C$ and $f$ on this band, the conclusion follows easily.\\

We are left with the case in which $I$ is bounded. In this case, $\gamma$ remains in a band for $t$ close enough to the limit; the completeness of $\gamma$ follows from the previous proposition. 
\end{proof}
\begin{remark}
This result holds for any extension $E_f^u$, with $f$ bounded. 
\end{remark}
\begin{proof}
	All we have to check is the completeness for a geodesic $\gamma$ that remains in a maximal ribbon, with unbounded $I$. We have from the previous proposition 
	$$t(x_1) - t(x_0) = \int_{x_0}^{x_1} \frac{dx}{\sqrt{C^2-\epsilon f(x)}}, \; x_1 > x_0 $$
	where $x_0$ is the coordinate of an orbit of $K^u$ cutting $\gamma$. 
	Since $f$ is bounded, we have
	$$t(x_1) - t(x_0) > \int_{x_0}^{x_1} \frac{dx}{\sqrt{C^2+M}},$$
	where $M$ is a majorant of $f$ on $\R$. Now, $I$ is assumed to be unbounded; therefore,  $x_1$ tends to infinity on $\gamma$, so that the above integral goes to infinity as $x_1 \to \infty$, which is the desired conclusion.		
\end{proof} 
\section{Study of the Jacobi equation}\label{Section équation de Jacobi}
Given a Lorentzian surface with a Killing vector field $K$, the Jacobi equation along a non-null geodesic $\gamma$ writes:
\begin{equation}\label{Jacobi}
u''+\epsilon \kappa u=0,
\end{equation}
where $\kappa(t)$ is the curvature along the geodesic $\gamma(t)$, $t \in I$. The restriction on the sign of the sectional curvature in the Riemannian setting leads to deep knowledge about the dynamics of the geodesic flow, through the behavior of the Jacobi fields. Unfortunately, this hypothesis makes no sense in the Lorentzian setting, since the quantity involved in the Jacobi equation in this case is $\epsilon \kappa$, where $\epsilon$ is the type of the geodesic.
So in this paragraph, we shall investigate some properties of the solutions of the differential equation (\ref{Jacobi}) under certain restrictive assumptions on the function $\kappa(t)$. The results will be applied in the case of a Lorentzian torus with a Killing vector field in the next paragraph. 
\begin{lemma}(see Remark 1.3 \cite{9}).\label{1zéro}
Given two independent solutions of the Jacobi equation, between two zeros of one solution, there is one and only one zero of the other. 
\end{lemma}
Now, denote by $s$ and $c$ the linearly independent solutions of (\ref{Jacobi}) satisfying 
$$\left\{
    \begin{array}{ll}
        s(0)=0 \\
       s'(0)=1
    \end{array}
\right.
\text{\;\;and\;\;}
\left\{
    \begin{array}{ll}
        c(0)=1 \\
       c'(0)=0
    \end{array}
\right.
$$
It's easy to check that $(cs^{'}-c^{'}s)(t)=(cs^{'}-c^{'}s)(0)=1$ for all $t \in I$.
Assume $\kappa(t)$ is periodic, of period $2\tau$. Denote by $A$ the endomorphism of the vector space $V$ of the solutions of (\ref{Jacobi}) (generated by $s$ and $c$), and given by  
\begin{align*}
A: V &\longrightarrow V \\
u(t) &\longmapsto u(t+2\tau)
\end{align*}
Equation (\ref{Jacobi}) has a periodic solution if and only if the endomorphism $A$ has eigenvalue $1$; in this case, the periodic solutions of the equation are exactly the eigenvectors associated with this eigenvalue.    \\
The zeros of a non-trivial solution $\phi$ of (\ref{Jacobi}) are simple, because $\phi(a)=\phi^{'}(a)=0$ for some $a$ implies $\phi =0$.
Now, let $\zeta_1 < \zeta_2$ be two consecutive zeros of a solution of (\ref{Jacobi}). By Lemma \ref{1zéro}, there is a unique zero $\zeta$ of $s$ in $[\zeta_1,\zeta_2]$. Denote by $P$ the set of all such pairs of zeros, for all the non-trivial solutions of (\ref{Jacobi}), and define 
$$P^{+}=\{(\zeta_1,\zeta_2) \in P, s^{'}(\zeta) >0 \} \text{\;\;and\;\;} P^{-}=\{(\zeta_1,\zeta_2) \in P, s^{'}(\zeta) <0 \}.$$ 
\begin{lemma}(Separation of zeros)\label{Le lemme des gendarmes}
Consider the differential equation (\ref{Jacobi}). Suppose we have the following conditions:\\
1) $\kappa(t)$ is even and periodic, of period $2\tau$. It is easy to check that under this condition, $s$ is odd and $c$ is even. \\
2) $s$ is periodic, of period $2\tau$. \\
Denote by $\alpha > 0$ a value of $s$. Define $$S_{\alpha}=\{t \in \R,\;\; |s(t)|=\alpha\}.$$ 
Suppose that $\alpha$ is such that $S_{\alpha}$ doesn't contain any zero of $s^{'}$; $S_{\alpha}$ is then a discrete set. 
Then:\\
- two consecutive zeros in $P^{+}$ of any solution of (\ref{Jacobi}) are separated by two distinct elements of $S_{\alpha}$ if and only if $c(t_0)+c(t_1) \geq 0$, where $t_0$, $t_1$ are the smallest positive consecutive reals in $S_{\alpha}$. \\
- two consecutive zeros in $P^{-}$ of any solution of (\ref{Jacobi}) are separated by two distinct elements of $S_{\alpha}$ if and only if $-c(t_0^{'})-c(t_1^{'})+ 2\frac{c^{'}(\tau)}{s^{'}(\tau)}s(t_0^{'}) \geq 0$, where $t_0^{'}$, $t_1^{'}$ are the biggest elements $<\tau$ in $S_{\alpha}$.
\end{lemma}
\begin{proof}
We start with the following observation: consider two independent solutions of (\ref{Jacobi}) with $a_1$ and $a_2$, $b_1$ and $b_2$ the two respective zeros in $[-\tau,\tau]$ such that $a_1<b_1$ and $a_2<b_2$; Lemma \ref{1zéro} ensures the existence of a one-to-one correspondence between $]a_1,b_1[$ and $]a_2,b_2[$ that sends $t \in ]a_1,b_1[$ to the unique zero in $]a_2,b_2[$ of the solution (taken up to a multiplicative constant) vanishing at $t$. \\

Since we have an action of the operator $A$ on the space $V$ of the solutions of (\ref{Jacobi}), it suffices to show this when $\zeta=0$ and $\zeta = \tau$. 
Denote by $\phi_0$ and $\phi_1$ the two solutions  vanishing at $-t_0$ and $-t_1$ respectively, defined by  
 $$\phi_i(t)=c(t)+\frac{c(t_i)}{s(t_i)}s(t), \;\; i=0,1.$$    
By definition, $s(t_0)=s(t_1)$. An easy computation then gives for $i =0,1$: 
\\
. $\phi_i(-t_i)=0$\\
. $\phi_i(0)=1$\\
. $\phi_i(t_{1-i})= c(t_0)+c(t_1)$.\\
\\
When $c(t_0)+c(t_1) \geq 0$, $\phi_i(0)$ and $\phi_i(t_{1-i})$ are both non-negative so that $\phi_i$ doesn't vanish on $[0,t_{1-i}[$ by Lemma \ref{1zéro}. Denote by $z_0$ and $z_1$ the zeros of $\phi_0$ and $\phi_1$ respectively in $[0,\tau]$; we thus have $z_0 \geq t_1$ and $z_1 \geq t_0$, the last inequalities being reversed when $c(t_0)+c(t_1)<0$. Hence, using the above observation, we deduce that when $c(t_0)+c(t_1) \geq 0$ is satisfied, two consecutive zeros in $[-\tau,\tau]$ of a solution of the Jacobi equation are always separated by two elements of $S_{\alpha}$. Now, suppose that $c(t_0)+c(t_1)<0$, one can find, using the same observation as above, a solution of (\ref{Jacobi}) vanishing at $z_0$ and $z_1$ in $[-\tau,\tau]$, such that $-t_0 < z_0 < 0 $ and $0 < z_1 < t_1$. \\

In the same way, two consecutive zeros in $[0,2\tau]$ of a solution of the Jacobi equation are separated by two elements of $S_{\alpha}$ if and only if $$\psi(2\tau -t_1^{'})+\psi(2\tau-t_0^{'}) \geq 0,$$ where $\psi$ is the solution satisfying $\psi(\tau)=1, \psi^{'}(\tau)=0$. This solution is given by $\psi(t)= \frac{1}{c(\tau)}(c(t) -\frac{c^{'}(\tau)}{s^{'}(\tau)} s(t))$. Now an easy computation gives $$c(2\tau-t)=c(t)-2\frac{c^{'}(\tau)}{s^{'}(\tau)}s(t).$$
Indeed, $c(2\tau-t)$ is again a solution of (\ref{Jacobi}): write $c(2\tau-t)=\lambda_1 c(t) + \lambda_2 s(t)$. Setting successively $t=0$ and $t=\tau$ yields $\lambda_1=c(2\tau)=1/s^{'}(2\tau)=1$, and $\lambda_2=-2 c^{'}(\tau)/s^{'}(\tau)$.\\
Now, observe that $c(\tau) < 0$ for, if not, $c$ would vanish twice in $[0,\tau]$, and this is not possible by lemma \ref{1zéro}. 
This ends the proof. 
\end{proof} 
\textbf{Notation:} Fix $\alpha > 0$ a value of $s$, and consider an interval $I$ of $\R$. The property that two consecutive zeros in $I$ (if they exist) of any solution of (\ref{Jacobi}) are separated by two distinct elements of $S_{\alpha}$ will be referred to as the property $(\mathfrak{P})$. 
\begin{definition}
We say that an interval $I = ]x,y[$ of $\R$ is a domino associated to $\alpha$ if $x, y \in S_{\alpha}$ and $I$ contains only one element of $S_{\alpha}$.
\end{definition}
Note that when $\alpha = |C|/\beta^{'}(0)$, where $C$ is the Clairaut constant related to $\gamma$, if $I$ is a domino of $\R$, the geodesic restricted to $I$ lies in a domino of the surface, in the sense of Definition \ref{Définition ruban/domino}.
\begin{corollary}\label{Corollaire des gendarmes}
Suppose the conditions of Lemma \ref{Le lemme des gendarmes} are satisfied, and that in addition, there exist exactly two elements of $S$ in $[0,\tau]$, hence between any two zeros of $s$. Keeping the same notations as in Lemma \ref{Le lemme des gendarmes}, we have:\\
- the domino containing $0$ has the property $(\mathfrak{P})$ if and only if $c(t_0)+~c(t_1)~\geq 0$,\\
- the domino containing $ \tau$ has the property $(\mathfrak{P})$ if and only if $-c(t_0)-~c(t_1)+~ 2\frac{c^{'}(\tau)}{s^{'}(\tau)}s(t_0) \geq~0$.
\end{corollary}
\begin{lemma}\label{Lemme beta distance minimale}
Consider the differential equation (\ref{Jacobi}). Suppose in addition to the conditions in Lemma \ref{Le lemme des gendarmes} that $\kappa$ is $\tau$ periodic and $s$ is $\tau$ antiperiodic, i.e. $$s(\tau+t)=-s(t), \forall t \in \R.$$ 
Then the domino containing $0$ has the property $(\mathfrak{P})$ if and only if $s$ realizes the minimum distance between the zeros of the solutions of the equation (\ref{Jacobi}).
\end{lemma}
\begin{proof}
An easy computation gives 
\begin{align}\label{Relation sur c, symétries ajoutées}
c(\tau -t)=2\frac{c(\tau/2)}{s(\tau/2)}s(t) - c(t) \;\; \forall t \in \R.
\end{align}
In this case, we have $t_1=\tau -t_0$ so that $c(t_0)+c(t_1) \geq 0$ reads $c(\tau/2) \geq 0$. \\
Let $z_0$ and $z_1$ be two any consecutive zeros of a solution of equation (\ref{Jacobi}). If $\phi$ is a solution of (\ref{Jacobi}), $\phi(.+\tau)$ is also a solution, so one can assume that $z_0 \in ]-\tau,0[$. Write
$$\phi(t)=c(t)-\frac{c(z_0)}{s(z_0)}s(t),$$
we have  $\phi(z_0)=0$, $\phi (0) =1$, and $\phi (z_0 + \tau)=-2 \frac{c(\tau/2)}{s(\tau/2)}s(z_0)$, by formula (\ref{Relation sur c, symétries ajoutées}) above.\\
We see that $\phi(0)$ and $\phi(z_0+\tau)$ have the same sign if and only if $c(\tau/2) \geq 0$. In this case, $\phi$ doesn't vanish on $[0,z_0+\tau[$ by Lemma \ref{1zéro}, and $z_1 \geq z_0+\tau$. This ends the proof.
\end{proof}

Lemma 3 in \cite{3} gives a result on the zeros of the solutions of the differential equation 
\begin{align}\label{eq diff dist min}
u^{''}(x)+p(x)u(x)=0,
\end{align} 
under certain restrictive assumptions on the function $p(x)$. Under these assumptions, the even solution of the equation realizes the minimum distance between such zeros. We shall adapt the lemma to the case where the coefficient is a periodic function; the same proof works in this case. 
\begin{lemma}\label{lemma dist min}
Let $p(x)$ have the following properties:\\
(a) \,$p$ is continuous;\\
(b) \,$p(-x)=p(x)$;\\
(c) \,the even solution of the differential equation (\ref{eq diff dist min}) vanishes for $\alpha >0$ and does not vanish for $-\alpha < x < \alpha$; \\
(d) \,$p(x)$ is non-increasing for  $0 < x < \alpha$;\\
(e) $p$ is periodic, with period $2 \alpha$.\\
Let $\alpha_1$, $\alpha_2$ be any pair of consecutive zeros of any non-trivial solution of (\ref{eq diff dist min}). Then $$\alpha_2 - \alpha_1 \geq 2\alpha.$$
\end{lemma}
\textbf{Note:} in \cite{3}, the condition (d) reads "$p(x)$ is non-increasing for $0 < x < +\infty$". When $p(x)$ is assumed to be periodic of period $2 \alpha$, it is enough to suppose this property on $]0,\alpha[$. 
\section{Conjugate points of Lorentzian tori with a Killing vector field}\label{Section tores SPC et géodésiques périodiques}
\subsection{Conjugate points from the Killing vector field}
Given a Lorentzian torus with Killing  field $K$, modeled on $E^u_f$, we have the following result from \cite{10}:
\begin{theorem}(Theorem 5.29 \cite{10}).\label{Conditions nécessaires SPC}
Let $f \in C^{\infty}(\R,\R)$ be a non-constant periodic function, and let $T$ be a torus modeled on $E_f^u$. If $T$ has no conjugate points, then: \\
(1) the set of connected components of $\{f \neq 0 \}$ is locally finite,\\
(2) $f$ changes sign between bands with a common boundary,\\
(3) $f'$ changes sign once in a band,\\
(4) each component defines a type II band in the torus.
\end{theorem}
Recall that a critical orbit of $K$ is an orbit corresponding to a critical point of $\langle K , K \rangle$. These leaves are geodesics (they are exactly the set of leaves of $K$ on which $\nabla_K K = 0$).
\begin{lemma}\label{orbites critiques SPC}
On a torus such that condition (3) is satisfied, the critical orbits of $K$ are without conjugate points.
\end{lemma}
\begin{proof}
The curvature is given by $\kappa=f^{''}(x)/2$ in the $x$-coordinate. From hypothesis (3), it follows that a local maximum of $f$ is necessarily positive and a local minimum is negative. This yields the following:
\begin{fact}\label{Fait sur le signe de k en les orbites critiques} The timelike critical orbit of $K$ corresponding to a local minimum of $f$ lies in a region of non-negative curvature, and the spacelike critical orbit (corresponding to a local maximum of $f$) lies in a region of non-positive curvature. On the other critical orbits of $K$, we have $\kappa =0$.
\end{fact}
It is then straightforward that equation (\ref{Jacobi}) does not have a solution vanishing twice on such geodesics. 
\end{proof}
As pointed out in the introduction of this paper, we restrict our attention to non-null geodesics. The geodesics perpendicular to $K$ are already known to be without conjugate points in the torus (see \cite[Proposition 1.8]{9}); this also holds for the geodesics such that $\beta$ does not vanish, by Lemma \ref{1zéro}. We are left with the case in which $\gamma$ is not perpendicular to $K$ and $\beta$ vanishes without being identically zero. We begin by investigating the conjugate points produced by $\beta$. \\

The geodesics are always thought of as being in the extension $E^u_f$. 

\begin{lemma}\label{Behavior of geodesics}
Let $\gamma$ be a non-null geodesic not perpendicular to $K^u$.\\
Assume $\gamma$ is not a critical geodesic of $K^u$. If the torus satisfies conditions (2)-(3) of Theorem \ref{Conditions nécessaires SPC} above, we have:\\
(i) $\beta$ does not vanish more than one time in a band. \\
(ii) a geodesic $\gamma$ tangent to $K^u$ in the torus leaves the band in which $\beta$ vanishes and crosses the type I neighboring band in $E^u_f$. So if the torus contains only type II bands, the geodesic lies in $\widetilde{T}$ exactly in a domino. 
\end{lemma}
\begin{proof}
(i) Let $p$ be a point where $\gamma$ is tangent to $K$ in the torus, and  call $B$ the band containing $p$.  Define $U:= \{z \in B, z \textrm{\; is an extremum of\;} \langle K,K \rangle \}$; condition (3) on $f$ implies that this set is connected, hence splits the band into two connected components.  In the interior of the band, $K$ has type $\epsilon$, so the function $\epsilon \langle K,K \rangle$ is positive on $B$. Suppose $\beta$ vanishes another time in $B$. It appears from the formula $\epsilon \langle K,K \rangle =C^2-\beta^2$ that $p$ is a local maximum of $\epsilon \langle K,K \rangle$ on $\gamma$, so if $q$ is the closest point to $p$ in which $\beta$ vanishes, $p$ and $q$ are necessarily in the same connected component, and in addition, $\epsilon \langle K,K \rangle$ is decreasing from $p$ to $q$. Now since $\langle K,K \rangle(p) =\langle K,K \rangle (q)=\epsilon C^2$, this means that $\epsilon \langle K,K \rangle$ is constant between the two points, thus $\gamma$ is a critical orbit of $K$, contrary to our assumption. \\
(ii) If $\gamma$ remains in $B$, it asymptotically approaches an orbit of $K$ contained in $B$ on which the norm of $K$ is the same as on the point where $\gamma$ is tangent to $K$ (see Lemma \ref{asymptote à K}); the same argument above shows that this is impossible. It follows that $\gamma$ leaves the band from both sides. Now, since $f$ changes sign between two consecutive bands, $\gamma$ cannot cross the type III neighboring band. Repeated application of Lemma \ref{asymptote à K} shows that $\gamma$  cannot remain in that  band either. In addition, $\gamma$ cannot cross a type II band unless it is perpendicular to $K$. So $\gamma$ crosses the type I neighboring band.  
\end{proof}

\begin{corollary}
If $(T,K)$ is a torus satisfying the hypotheses in Theorem \ref{Conditions nécessaires SPC}, then the Killing vector field doesn't produce conjugate points in the torus.
\end{corollary}
\begin{proof}
Let $\gamma$ be a geodesic, lifted and then extended to $E^u _f$, such that $\beta$ vanishes at least twice. We have to prove that we cannot have two zeros of $\beta$ in the torus.  This follows from the fact that the geodesic crosses a type I band between two such zeros as a consequence of Lemma \ref{Behavior of geodesics} above.  When the torus is assumed to have only type II bands, these zeros are never both in the torus. 
\end{proof}
\subsection{Invariant geodesics and conjugate points}\label{Section pour définir Z}
We begin by stating a lemma that will be used at the end of this section. Choose an orientation on the torus, and let $p \in T$ be a point such that $\langle K(p),K(p) \rangle =0$. If the orbit of $K$ containing $p$ belongs to the first line of the null cone bordering the negative cone, fix a null vector field $L$ in the maximal ribbon of $\widetilde{T}$ containing $p$ such that $\langle L,K \rangle =-1$; otherwise, take $L$ such that $\langle L,K \rangle =1$. This way, we get a local coordinate $x$ in the ribbon containing $p$, which we globalize to $E^u_f$ in a unique way. 
\begin{lemma}\label{Formule sur beta' en fonction de f}
	Let $\gamma$ a geodesic in $E^u_f$, then  
	$$\epsilon \beta^{'}(t) = \frac{1}{2} f^{'}(x(\gamma(t)).$$
\end{lemma}
\begin{proof}
	We can assume that the orbit of $K$ containing $p$ belongs to the first line of the null cone bordering the negative cone. In the ribbon containing $p$, call it $U$, set $K = C T + \beta N$ and $L=\lambda_1 T + \lambda_2 N$; $\lambda_1$ and $\lambda_2$ are never vanishing functions on $\gamma$. In fact, the choice of $L$ leads to $\lambda_1 = -\epsilon \lambda_2$. We have $\nabla_T K = \beta^{'} N$ and $ \nabla_N K = \beta^{'} T$. An easy computation then gives $\nabla_L K = \epsilon \beta^{'} L$, yielding   $$\frac{1}{2} \nabla_L \langle K , K \rangle = \epsilon \beta^{'}.$$ 
	Now in the ribbon obtained from $U$ by a generic reflection, the local coordinate is defined by the null vector field $L^{'}$ such that $\langle L^{'},K \rangle=1$, and $L^{'}$ belongs to the other null line field. We can check that this gives the same formula, finishing the proof. 
\end{proof}

\begin{corollary}\label{Asymptote à une orbite critique de K}
	Let $(T,K)$ be a torus modeled on $E^u_f$, with $f$ a non-constant periodic function such that $f^{'}$ changes sign once in a band. Then a non-null geodesic $\gamma$ that lies in a band provided $|t|$ is large enough, asymptotically approaches a critical orbit of $K^u$, and this orbit is either timelike or spacelike, depending on the type of $\gamma$.
\end{corollary}
\begin{proof} 
	Provided $|t|$ is large enough, $\gamma$ is transverse to $K$ (see Lemma \ref{Behavior of geodesics}, (i)), so an application of Lemma \ref{asymptote à K} shows that $\gamma$ asymptotically approaches a leaf of $K^u$ whose type is the same as $\gamma$'s, i.e. $x^{'}$ goes to $0$ on $\gamma$, hence 
	\begin{align}\label{beta tend vers 0, ici}
	 \displaystyle \lim_{t \rightarrow \infty} \beta (t) =0,
	\end{align}
	since $x^{'}(t)^2=\beta(t)^2$. On the other hand, $\epsilon \beta^{'}(t)$ converges by Lemma \ref{Formule sur beta' en fonction de f} above; combining this with (\ref{beta tend vers 0, ici}), we see that this limit is necessarily zero.
\end{proof}
Let $\gamma$ be a non-null geodesic such that $\beta$ vanishes without being identically zero, paramatrized so that $\beta(0)=0$. 
\begin{lemma}\label{lemme géod périodique}
Assume $\beta$ vanishes twice and let $\omega$ be the half-distance between two consecutive zeros (measured in the parameter $t$ of the geodesic). Then $\beta$ is odd and periodic, of period $4\omega$ (in particular, the distance between two consecutive zeros of $\beta$ is always $2\omega$). Furthermore, $\kappa$ is a $4\omega$-periodic and even function. 
\end{lemma}
\begin{proof}
The proof is based on a powerful geometric ingredient of Lorentzian tori with a Killing vector field: the generic reflection that fixes a non-degenerate  leaf of $K^{\perp}$, defined on the saturation of the leaf by the flow of $K$, extends to a  global isometry of the extension (see Proposition \ref{Extension des réflexions}). A point where $\beta$ vanishes is either a saddle point or corresponds to a point where $\gamma$ is tangent to $K$, depending on whether the geodesic is orthogonal to $K$ or not. The behavior being slightly different in each case, we choose to consider them separately. Suppose $\gamma$ is not a leaf of $K^{\perp}$; the reflection that fixes the geodesic perpendicular to $K$ that passes through $\gamma(0)$ is an isometry that preserves $\gamma$ (actually, it reverses the orientation), sending $K$ to $-K$; it sends  $\gamma(t)$ to $\gamma(-t)$ yielding for all $t$,\; $\beta(-t)=-\beta(t)$. Composing two reflections that correspond to two consecutive zeros of $\beta$ gives an isometry preserving $\gamma$, which is actually a translation of the parameter $t$ by $4 \omega$. It sends $\gamma(t)$ to $\gamma(t+4\omega)$, and gives for all $t$, $$\beta(t+4 \omega)=\beta(t),$$
as expected. Now suppose that $\gamma$ is perpendicular to $K$; $\beta$ vanishes at the saddle points. Let $p$ be a saddle point on $\gamma$.
According to \cite{10}, we might choose two generic reflections each one fixing a leaf of $K^{\perp}$ passing through $p$, such that their composition is the reflection with respect to $p$. It is not hard to see that the obtained reflection is an isometry preserving $\gamma$; it again reverses its orientation and sends $K$ to $-K$. Composing two such reflections as before (corresponding to two saddle points), the same conclusions follow.  \\
The same arguments provide the proof of the statement on $\kappa$.
\end{proof} 
\begin{definition}\label{déf periodic geod}
A geodesic of $E^u_f$ with $\beta$ a non-zero periodic function is called "$\beta$-periodic" or an "invariant geodesic", for they are invariant by an isometry of the extension.
\end{definition}
The term "$\beta$-periodic" will be adopted below because it seems to be more suggestive of the type of resulting Jacobi equation. We simply say "periodic geodesic" instead of "$\beta$-periodic". 

If we suppose that $f$ has simple zeros, the geodesics perpendicular to $K$ are all periodic.
\begin{corollary}\label{Corrolaire: deux comportements distincts, périodique ou asymptote}
Let $\gamma$ be a non-null geodesic tangent to $K$ in the torus; we distinguish two different behaviors of $\gamma$ in $E^u_f$: either $\gamma$ is periodic, or it asymptotically approaches a critical orbit of $K$ with infinite $t$; in this case, the geodesic is tangent to $K$ only one time.
\end{corollary}
\begin{proof}
If $\gamma$ cuts a null orbit of $K^u$ denoted by $\mathfrak{l}$, we claim that $\gamma$ crosses a type I band $B$ containing $\mathfrak{l}$ if and only if $C^2 > \sup_B \epsilon \langle K^u,K^u \rangle$. Indeed, if $\gamma$ crosses a type I band, it is everywhere transverse to $K^u$ in that band, for if $\beta$ vanishes, either $\gamma$ remains in the band or it crosses a type III band containing $\mathfrak{l}$. Writing $C^2 - \beta^2 = \epsilon \langle K^u,K^u \rangle$, we see that $C^2 > \sup_B \epsilon \langle K^u,K^u \rangle$. Now, assume this inequality holds, then the same argument as in the proof of Lemma \ref{Behavior of geodesics}, (iii), shows that $\gamma$ crosses the type I band $B$.  In fact, as long as this inequality holds, $\gamma$ remains in the ribbon containing $\mathfrak{l}$, by cutting the leaves of $K^u$ transversally.\\ 
Now, if $\gamma$ enters a band where  $C^2$ is reached by $\epsilon \langle K^u, K^u \rangle$ in that band, using Corollary \ref{Asymptote à une orbite critique de K} and the previous lemma, we distinguish two different cases: \\
\;\;\;\; a) $C^2$ is not a critical value of $f$; in this case, $\gamma$ meets tangentially the closest orbit of $K^u$ to the boundary of the band meeting $\gamma$, on which $C^2 = \epsilon \langle K^u,K^u \rangle$, so that $\beta$ vanishes a second time. The geodesic is thus periodic. \\
\;\;\;\; b) $C^2$ is a critical value of $f$, which amounts to saying that the corresponding orbit of $K^u$ is a geodesic. In this case, if the closest orbit of $K^u$ to the boundary meeting $\gamma$ is not a geodesic, then $\gamma$ behaves as in a); otherwise, it approaches the orbit asymptotically, with an infinite $t$. 
\end{proof}
\begin{remark}
	This corollary holds for every torus modeled on $E^u_f$ such that $f$ is a non-constant periodic function. When assumptions (2)-(3) in Theorem \ref{Conditions nécessaires SPC} hold, the periodic geodesics leave every maximal ribbon in $E^u_f$, by crossing a type I band between two zeros of $\beta$. 
\end{remark}  
\begin{definition}
Define $\mathcal{L}^*_K(T)$ as the subset of $\mathcal{L}_K(T)$ such that the function $f$ induced by the norm of $K$ has simple zeros and satisfies conditions (3)-(4) of Theorem \ref{Conditions nécessaires SPC}. In particular, the assumption "$f$ has simple zeros" implies conditions (1)-(2) of Theorem \ref{Conditions nécessaires SPC}.
\end{definition}
Given a torus in $\mathcal{L}^*_K(T)$, using the results in section \ref{Section équation de Jacobi}, we shall give a necessary and sufficient condition for geodesics (other than the critical orbits of $K$) tangent to $K$ in the torus to be without conjugate points in it.\\
Recall that when assumptions (2)-(3) in Theorem \ref{Conditions nécessaires SPC} hold, if $\gamma$ is a geodesic tangent to $K$ at $p \in T$, and not a critical orbit of $K$, then $p$ is the unique point of $T$ where $\gamma$ is tangent to $K$.
\begin{proposition}(Characterization of geodesics without conjugate points) \label{SPC}
Let $(T,K)$ be a torus modeled on $E^u_f$, with $f$ satisfying the properties in Theorem \ref{Conditions nécessaires SPC}. Let $\gamma$ be a geodesic tangent to $K$ in the torus; assume $\gamma$ is not a critical orbit of $K$ and set $t=0$ at this point. Then $\gamma$ is without conjugate points in the torus if and only if $c(t_0)+c(t_1) \geq 0$, where $t_0$, $t_1$ are the smallest positive consecutive reals such that the norm of $K$ vanishes on $\gamma$.  
\end{proposition}
\begin{proof}
The two reals $t_0$ and $t_1$ do exist by Lemma \ref{Behavior of geodesics}. The geodesic lies in the torus in a domino, either for $t \in [-t_1, t_0]$ or for $t \in [-t_0,t_1]$, $t_0$ and $t_1$ being as in the statement above. Assume $\gamma$ is periodic; if $c(t_0)+c(t_1) \geq 0$, the intervals $[-t_1, t_0]$ and $[-t_0,t_1]$ satisfy the property $(\mathfrak{P})$ for $\alpha= |C|$, by Lemma \ref{Le lemme des gendarmes}. All we have to show is that this implies that $\gamma$ is without conjugate points in the torus. Setting $\alpha= |C|$, $S_{\alpha}$ is the set of points in which the norm of $K$ vanishes. Assume $(\mathfrak{P})$ is satisfied on $[-t_1, t_0]$, i.e. two zeros of any solution of (\ref{Jacobi}) in $[-t_1, t_0]$ (or in $[-t_0, t_1]$) are separated by two zeros of $\langle K , K \rangle$.  Since $\gamma$ lies in $T^2$ in a domino, it contains only one zero of $\langle K , K \rangle$ in the torus. Therefore, two zeros of a solution of (\ref{Jacobi}) cannot be both in the torus. Now, suppose that $c(t_0)+c(t_1)<0$; one can find by Lemma \ref{Le lemme des gendarmes} a Jacobi field along $\gamma$ vanishing at $z_0$ and $z_1$ in $[-2\omega,2\omega]$, such that $-t_0 < z_0 < 0 $ and $0 < z_1 < t_1$. \\
The case in which $\gamma$ asymptotically approaches a critical orbit of $K$ works in much the same way by proving, following Lemma \ref{Le lemme des gendarmes}, that the property ($\mathfrak{P}$) is again equivalent to having $c(t_0)+c(t_1) \geq 0$. The detailed verification is left to the reader. 
\end{proof} 
\begin{theorem}\label{Fermé d'intérieur non vide}
Let $(T, K)$ be a torus in $\mathcal{L}^*_K(T)$ modeled on $E^u_f$. Denote by $(x_n)_{n \in \Z}$ the sequence of zeros of $f$, the $x_n$'s being taken in increasing order. If the torus has no conjugate points, then $$f^{'}(x_n) + f^{'}(x_{n+1}) = 0,\;\; \forall n \in \Z.$$ 
\end{theorem}
\begin{proof}
Denote again by $(0),..,(\mathfrak{n})$ the distinct bands of the torus, $(i)$ and $(i+1)$ having a common boundary and opposite signs. Denote by $x_i < x_{i+1}$ the zeros of $f$ in the band $(i)$. \\
Let $\gamma$ be a geodesic tangent to $K$ in a band $(i)$ of the torus. Set $t=0$ at this point. For $|C|$ small enough,  the geodesic crosses one type I band before $\beta$ vanishes again in a copy of the band $(i+2)$ of the ribbon containing $(i)$. The torus contains an isometric image of the geodesic $\gamma$ restricted to a domino containing the zero of $\beta$ in $(i)$, and an isometric image of the geodesic restricted to a domino containing the zero of $\beta$ in $(i+2)$. By Corollary \ref{Corollaire des gendarmes} and Lemma \ref{SPC}, these geodesics are without conjugate points in the torus  if and only if $c(t_0)+c(t_1) \geq 0$ and $-c(t_0)-c(t_1)+ 2\frac{c^{'}(2\omega)}{s^{'}(2\omega)}s(t_0) \geq 0$. Consider a sequence of geodesics $\gamma_n$ approaching a non-degenerate geodesic $\gamma_{\infty}$ perpendicular to $K$. Then, $|C|$ tends to $0$, $t_0^n \to 0$, and $t_1^n \to 2 \omega_{\infty}$, where $\omega_{\infty}$ is associated to $\gamma_{\infty}$.  We admit that $c(t_0)+c(t_1)$ depends continuously on $\gamma$ (this will be proved in section 5). Letting $|C|$ go to $0$, we have $$c_n(t_0^n)+c_n(t_1^n) \to c_{\infty}(0)+c_{\infty}(2\omega_{\infty})$$ $$\text{\;\;and\;\;} -c_n(t_0^n)-c(t_1^n)+ 2\frac{c_n^{'}(2\omega_n)}{s_n^{'}(2\omega_n)}s_n(t_0^n) \to - c_{\infty}(0)-c_{\infty}(2\omega_{\infty}),$$
where the solution $c_{\infty}$ associated to $\gamma_{\infty}$ is then evaluated at two consecutive saddle points of $\gamma_{\infty}$. Denote them by $p_i$ and $p_{i+1}$. 
On the other hand, $$c_{\infty}(0)+c_{\infty}(2\omega)=\frac{1}{s^{'}_{\infty}(0)} + \frac{1}{s^{'}_{\infty}(2\omega)}=\frac{\beta^{'}_{\infty}(p_i)+\beta^{'}_{\infty}(p_{i+1})}{\beta^{'}_{\infty}(p_{i+1})}.$$
We see that if the torus is without conjugate points, $c_{\infty}(0)+c_{\infty}(2\omega)$, and then $\beta^{'}_{\infty}(p_i)+\beta^{'}_{\infty}(p_{i+1})$ should be zero. Using Lemma \ref{Formule sur beta' en fonction de f}, this yields $f^{'}(x_i) + f^{'}(x_{i+1}) = 0$, for all $i$.
\end{proof}

\begin{remark}\label{Orbites incomplètes=zéros simpes}
	The geodesic parametrization of a null orbit of $K$ is incomplete if and only if it corresponds to a simple zero of $f$.
\end{remark}
\begin{remark}
	 When the geodesic is incomplete, it is easy to see that a geodesic	 parametrization is given by $-2 (\eta f^{'})^{-1} e^{-\frac{1}{2} \eta f^{'} t}$, where $K=\partial_t$, and $\eta=\pm 1$ depends on the null orbit of $K$ and the choice of the $x$-coordinate (once the coordinate is fixed, two consecutive null orbits of $K$ in the torus give $(\eta_1,\eta_2)$, such that $\eta_1+\eta_2=0$); therefore,  Theorem \ref{Fermé d'intérieur non vide}  above can  be stated in a more geometric way, as in Theorem \ref{Condition sur les lamdas, introduction}   in the introduction of this paper.
\end{remark}
\begin{corollary}\label{limite de tores APC}
	$T_{CP}$ can be obtained as a limit of Lorentzian tori with conjugate points, and admitting a Killing field.
\end{corollary}
\begin{remark}
		In \cite{9} it is shown that the space of Lorentzian tori without conjugate points is a closed subset of $\mathcal{L}(T)$ endowed with the $C^{\infty}$ topology, and that the Clifton-Pohl torus is on the boundary of this set. In the proof, the construction of the sequence of such metrics converging to $T_{CP}$ killed the Killing vector field. 
\end{remark}
\subsection{A class of Lorentzian tori without conjugate points}\label{les exemples de tores SPC}
The local geometry of a Lorentzian torus with a Killing vector field is determined by the function $f$ induced by the norm of $\widetilde{K}$ in the $x$-coordinate of the universal cover. In privileging this point of view, we obtain a family of Lorentzian tori without conjugate points admitting a Killing vector field with pairwise non-isometric universal cover; these examples include the Clifton-Pohl torus and quadratic variations of it. \\ 
Our examples are in $\mathcal{L}^*_K(T)$. In addition, we  assume that $f$ has two zeros in its smallest period, and that $\exists a \in \R, f(a+t)=f(a-t)$ (in particular, $\mathfrak{n}=2$). We then get additional isometries in $E^u_f$, known as non-generic reflections in \cite{10}, that lead to additional symmetries on the $\beta$ solution. In this case, given a non-null geodesic $\gamma$, not perpendicular to $K$, the behavior in b), Corollary \ref{Corrolaire: deux comportements distincts, périodique ou asymptote}, does not appear, and $\beta$ vanishes (without being identically zero) if and only if the geodesic is periodic. With the additional assumption on $f$, $\kappa$ has two symmetries, one of which about $t=0$ and the other about  $t=\omega$, hence $\kappa$ is $2 \omega$ periodic, and $\beta$ satisfies the following: 
$$\beta(2\omega -t)=\beta(t) \;\;\forall\, t \in \R.$$
In this case, Lemma \ref{Lemme beta distance minimale} together with Proposition \ref{SPC} imply that the absence of conjugate points in the torus for such geodesics is equivalent to the fact that $\beta$ realizes the minimum distance between the zeros of the solutions of the Jacobi equation. 
A large class of Lorentzian tori with no conjugate points is given in the following theorem:   
\begin{theorem}\label{famille SPC}
Let $f$ be a periodic function that satisfies the following properties:\\
i) $f$ has simple zeros,\\
ii) $f^{'}$ changes sign one time in a band,\\
iii) $f^{'}.f^{'''} \leq 0$,\\
iv) $\exists a \in \R, f(a+t)=f(a-t)$,\\
v) $f$ has two zeros in the smallest period of $f$.\\
Then, a torus modeled on $E^u_f$ and all of whose bands are of type II has no conjugate points.   
\end{theorem}
\begin{proof}
Let $\gamma$ be a non-null geodesic, and assume that $\beta$ vanishes. Either $\beta$ is identically zero, in which case the geodesic is a critical orbit of $K$ without conjugate points (Lemma \ref{orbites critiques SPC}), or it isn't identically zero, then $\gamma$ is periodic. The geodesic therefore has no conjugate points in the torus if and only if $\beta$  realizes the minimum distance between the zeros of the non-trivial solutions  of the Jacobi equation. Put $p(t)=\epsilon \kappa(t+\omega)$ and consider the differential equation  
\begin{align}\label{Ici seulement}
u^{''}(t)+\epsilon \kappa(t+\omega) u(t)=0
\end{align}
We have the following:\\
i)\,$p$ is even; this is a consequence of $\kappa$ being even and $2 \omega$ periodic. Indeed, write  $p(-t)=\epsilon \kappa (-t+\omega)=\epsilon \kappa(t-\omega)=\epsilon \kappa(t+\omega)=p(t)$;\\
ii)\,the even solution is given by $\beta(t+\omega)$;\\ 
iii)\,$p(t)$ is decreasing for $0 < t < \omega$.\\
To see this, denote $x_1 < x_2$ two  consecutive critical points of $f$ which correspond to extremal values. Condition (ii) implies that $f^{'}$ does not change sign on $]x_1,x_2[$. Since the curvature is given by $f^{''}(x)/2$ in the $x$ coordinate, the condition $f^{'}.f^{'''} \leq 0$ implies that on a ribbon, $\kappa$ is either  decreasing or increasing between the two orbits of $K$ corresponding to $x_1$ and $x_2$, depending on whether  $f$ reaches its minimal value on $x_1$ or $x_2$.   
Now the geodesic is transverse to $K$ over $]0,2\omega[$ and lies in a ribbon of the Lorentzian surface $E^u_f$. Since $\mathfrak{n}=2$, the geodesic crosses only one extremal orbit of $K$ between $0$ and $2 \omega$, the one corresponding to $t=\omega$. We then get that $\kappa(t)$ is decreasing (resp. increasing) for $\omega < t < 2 \omega$ if the geodesic is spacelike (resp. timelike), implying condition (iii). \\
Given Lemma \ref{lemma dist min}, the even solution $\beta(t+\omega)$ then realizes the minimum distance between the zeros of the non-trivial solutions  of equation (\ref{Ici seulement}). Therefore, when the torus has only type II bands, these geodesics have no conjugate points in the torus. This completes the proof of Theorem \ref{famille SPC}.  
\end{proof}

\subsection{Some examples of Lorentzian tori without conjugate points}
\begin{definition}
We call a quadratic variation of the Clifton-Pohl torus a metric of the form
$$g=\frac{2dxdy}{Q(x,y)},$$
where $Q$ is a positive definite quadratic form of determinant $1$.
\end{definition}
\begin{proposition}
The quadratic variations of the Clifton-Pohl torus are contained in the family obtained in Theorem \ref{famille SPC}.
\end{proposition}
\begin{proof}
Considering $-g$ instead of $g$ if necessary, we can assume that the quadratic forms are given by
$$Q(x,y)=u x^2 + 2v x  y + w y^2, \textrm{\;\;with\;\;} u>0, w>0, uw-v^2=1.\\$$	
Write $Q(x,y)=(\sqrt{u}x+\frac{v}{\sqrt{u}} y)^2 +(\frac{1}{\sqrt{u}}y)^2$. Applying the change of variables $x^{'}=\sqrt{u}x, y'=\frac{1}{\sqrt{u}}y$, we reduce to the metrics of the form 
$$g=\frac{2dxdy}{x^2+(y+ax)^2 },$$  
where $a\in \R$. These metrics admit a Killing vector field given by $K=x\partial_x + y \partial_y$. \\
We have now $$\kappa=-2K.K - 2a.$$
Hence, $f$ is a solution of the differential equation $$f^{''}(x)+4f(x)-4a=0,$$ 
and it is given by $f(x)= \sin(2x) -2 a\cos^2(x)$.\\
The Clifton-Pohl torus and its quadratic variations are then contained in the family obtained before, and are therefore without conjugate points.
\end{proof}
Other examples of Lorentzian tori without conjugate points, belonging to this family, are given at the end of this paper.
\section{Deformation of a Lorentzian torus without conjugate points}\label{Section stabilité par déformation}
Recall that a torus in  $\mathcal{L}_K(T)$ without conjugate points satisfies conditions (1)-(4) in Theorem \ref{Conditions nécessaires SPC}. So we defined $\mathcal{L}^*_K(T)$ to be the subset of Lorentzian tori in $\mathcal{L}_K(T)$ containing only type II bands (condition (4)), and modeled on $E^u_f$ such that $f$ has simple zeros and satisfies the third condition of the theorem (the condition on the zeros implies (1)-(2)). A torus without conjugate points in  $\mathcal{L}^*_K(T)$ satisfies in addition  the property in Theorem \ref{Fermé d'intérieur non vide}, which is a kind of "pointwise symmetry" on $f$; in the sequel we denote by $S\mathcal{L}^*_K(T)$ the set of Lorentzian metrics in $\mathcal{L}^*_K(T)$ satisfying this condition.\\
The family introduced in Theorem \ref{famille SPC} gives examples of metrics in $\mathcal{L}^*_K(T)$ without conjugate points; we wish to know if this property is stable by deformation in $S\mathcal{L}^*_K(T)$. Our main concern in this section will be to obtain stability criteria for these metrics, i.e. conditions on $f$ from which conclusions may be drawn as to the stability character of the property.    

This section contains 5 paragraphs. In Paragraph \ref{Section: convergence des champs de Killing}, we prove that for metrics in $\mathcal{L}_K(T)$, the Killing field depends smoothly on the metric. This allows us to define in Paragraph \ref{Section: A continuous function controlling conjugate points} a continuous function depending on the metric, that controls the conjugate points. Recall that the universal cover of a Lorentzian torus admitting a Killing field appears as an open subset of a bigger surface which contains conjugate points. For the stability question, a special interest will be given in paragraphs \ref{Section: Digression: Jacobi equations all of whose solutions are periodic} and \ref{Application: Invariant geodesics with non-periodic Jacobi vector field} to the case in which conjugate points are on the boundary. We will give sufficient conditions on the metric to avoid this situation, and obtain in the last paragraph a stability result.
\subsection{On the Killing vector field of a Lorentzian metric on $T$}\label{Section: convergence des champs de Killing} 
The space of smooth Lorentzian metrics on a torus is equipped with the $ C^{r} $ topology; since the torus is compact, this space is metrisable. If $K$ is a Killing field for a metric $ g $, then $ \lambda K $, $ \lambda \in \R^{*} $, is also a Killing vector field for $g$. Actually, if a (non-flat) metric on a torus admits a Killing field, then the latter is unique up to a multiplicative constant; this was proved in \cite{10}. To fix one, one fixes an orientation $\nu$ on the torus and a vector field $J$ everywhere transverse to $K$, together with a point $ p $ on $ T $ in which $ g(K,K)(p) \neq 0 $, and takes $\lambda_0 K$ such that $g(\lambda_0 K,\lambda_0 K)(p)=\eta$, $\eta=\pm 1$, and $\nu(K,J) > 0 $ (the latter is possible, for the foliation of $K$ is orientable). We simply denote it by $K$. This determines $ K $ on the orbit of $K$ containing $ p $, and on a null geodesic passing through this point (Clairaut's constant). This way, one  determines $K$ on the saturation of the geodesic by the flow of $ K $, i.e. on a ribbon; doing it on the ribbons in turn, $ K $ is determined on the whole torus. In the following, $K$ is fixed this way. \\
 
Now consider a sequence of Lorentzian metrics $g_n \in \mathcal{L}_K(T)$ which converges to a non-flat metric $ g \in \mathcal{L}_K(T)$ in the $C^r$ topology, $r \geq 4$; denote by $K$ the Killing vector field of $g$, $K_n$ that of $g_n$ for all $ n $. A natural question is to see whether $ K_n $ converges to $ K $. For convenience,  we start by examining the foliations associated to these vector fields. We want to provide the space of foliations of the torus, denoted by $\mathcal{F}(T)$, with a "natural" topology. The line bundle over a torus being a trivial bundle, we choose a trivialization and we associate to a foliation its tangent field, which is then a smooth map from $T$ into $P^1$. One identifies $ \mathcal{F}(T) $ with its image in $C^{r}(T,P^1)$ equipped with the $C^{r}$ topology, and gets a topology on $\mathcal{F}(T)$; this topology is independent of the choice of the trivialization. 

\begin{proposition}\label{Convergence champs Killing}
If $g_n$ is a sequence of Lorentzian metrics in $\mathcal{L}_K(T)$ converging to a non-flat metric $g \in \mathcal{L}_K(T)$ for the $C^r$ topology, $r \geq 4$, then the sequence of foliations associated to $K_n$ converges to that of $K$, for the $C^{r-3}$ topology. Then, if $p$ is a point on $T$ where $g(K,K)(p) =\eta$, $\eta=\pm 1$, then for $ n $ sufficiently large, $ K_n(p) $ has type $\eta$;\; fix \;$ K_n $ by setting\; $ g_n(K_n,K_n)(p) = \eta $, such that $K_n(p)$ converges to $K(p)$. The sequence $ K_n $ so obtained converges to $ K $ for the $C^{r-4}$ topology.
\end{proposition}
In the following, $g \in \mathcal{L}_K(T)$ is a non-flat metric; we denote by $U$ the open set in which $\kappa$, the sectional curvature associated to $g$, is a submersion. On this set, $\kappa$ defines the foliation associated to $K$. The open set of $T$ not containing the null orbits of $K$ will be denoted by $V$. 
\begin{fact}\label{Fait1: Convergence des feuilletages}
Consider a sequence of metrics $g_n \in \mathcal{L}_K(T)$ converging to $g$ for the $C^r$ topology; then the sequence of foliations associated to $K_n$ converges to that of $K$ for the $C^{r-3}$ topology, on every compact subset of $U$.	
\end{fact}
\begin{proof}  
Observe that the space of submersions over a compact variety $M$ is an open set in the space of $C^1$ functions over $M$. Let $B$ be a compact subset of $U$; provided we take $ n $ large enough, $ \kappa_n $ is also a submersion on $ B $ that defines the foliation associated to $ K_n $.\\
In smooth local coordinates $ (s,t) $, define $$ X _ {\kappa} (s, t): = \Big{(} - \frac{\partial \kappa}{\partial t}, \frac{\partial \kappa}{\partial s} \Big{)}, $$
a vector field associated to $ \kappa $. We have $ d \kappa (X_{\kappa}) = 0$, so $ X_{\kappa} $ defines the foliation of $ K $. It follows that if $ \kappa_n $ converges to $ \kappa $ for the $C^{k-2}$ topology, the foliations they define converge for the $C^{k-3}$ topology. Note that the curvature defines the foliation but not the Killing field; indeed, if $ X $ is such that $ d \kappa(X) = 0 $, then $ d \kappa(\phi X) = 0 $ for any function $ \phi $ on the torus. 
\end{proof}
\begin{fact}\label{Fait2: convergence de K_n à partir des feuilletages}
Denote by $\mathcal{F}$ (resp. $\mathcal{F}_n$) the foliation associated to $K$ (resp. $K_n$). Suppose there exists $q \in V$ such that $K_n(q)$ converges to $K(q)$. Fix $K$ and $K_n$, for $n$ big enough, by setting $g(K,K)(q)=g_n(K_n,K_n)(q)=\eta$, $\eta=\pm 1$.  If $\mathcal{F}_n$ converges to $\mathcal{F}$ on $V$ for the $C^k$ topology, with $k \leq r$, then $K_n$ converges to $K$ on every compact subset of the connected component of $q$ contained in $V$. The convergence holds for the $C^{k}$ topology. 	
\end{fact}
\begin{proof}
Fix a Riemannian metric $S$ on $T$, and let $X$ (resp. $X_n$) be a vector field on $T$ tangent to $\mathcal{F}$ (resp. $\mathcal{F}_n$), such that $X(q)$ (resp. $X_n(q)$ is $\R^+$-collinear to $K(q)$ (resp. $K_n(q)$, and $S(X,X)=1$ (resp. $S(X_n,X_n)=1$. Since the foliations are orientable, the oriented foliations still converge, so that $X_n$ converges to $X$ for the $C^k$ topology. Next, we define $\bar{K}$ (resp. $\bar{K}_n$) to be the vector field on the connected component of $V$ containing $q$, $\R^+$-collinear to $X$ (resp. $X_n$) such that  $ g(\bar{K}, \bar{K}) = \epsilon = g_n(\bar{K}_n, \bar{K}_n) $, for $ n $ big enough. We have $\bar{K}_n \overset{C^k}{\to} \bar{K}$.  
Write $ K = h \bar{K}$ and $K_n = h_n \bar{K}_n, $ where $ h, $\,$ h_n> 0 $, for all $ n $. This yields $ h= \sqrt{\epsilon g(K,K)}$, which implies that $ h $ (resp. $h_n$) is invariant by the flow of $ K $ (resp. $K_n$).\\
Now, write
$$ g(\nabla_X K , X) = 0, \text{\;\;} \forall X. $$
This yields the following equation on $ h $ 
\begin{align}\label{Eq diff pour h}
(\nabla_X h) g( \bar{K} , X) + h \text{\;} g(\nabla_X \bar{K}, X) =0, \text{\;\;} \forall X.
\end{align}
Let $ c(s) $ be a curve containing $q$ and transverse to $K$, and define a $C^r$-diffeomorphism $\phi$ by setting $\phi(s,t)=F(t,c(s))$, where $F$ is the flow of $\bar{K}$. There exists an open neighborhood of $\R^2$ on which the flows of $\bar{K}_n$ are all defined, for $n$ big enough. So define in the same way a diffeomorphism $\phi_n$ by $\phi_n(s,t)=F_n(t,c(s))$, where $F_n$ is the flow of $\bar{K}_n$. \\
Since $\bar{K}_n$ is $C^k$ close to $\bar{K}$, for $n$ large enough,  $\phi_n$ is $C^k$ close to $\phi$. Set $X=\frac{\partial \phi}{\partial s}$ and $X_n=\frac{\partial \phi_n}{\partial s}$ in equation (\ref{Eq diff pour h}) above; combining  this with the fact that $h$ is invariant by $\bar{K}$,  $h\, o\, \phi$, written in the $(s,t)$ coordinates, satisfies the following differential equations:
\begin{align*}
\frac{\partial (h o \phi)}{\partial t}&  =0,\\
\frac{\partial (h o \phi)}{\partial s}& g(\bar{K} , X)\, o\, \phi + (h o \phi) g(\nabla_X \bar{K} , X) \, o \, \phi =0.
\end{align*}
This yields $h o \phi (s,t) = h o \phi (s,0)$ and  $h_n o \phi_n (s,t) = h_n o \phi_n (s,0)$, for all $t$. The continuity of solutions in the initial conditions and the coefficients of the equation permits now to assert that $h_n o \phi_n (s,0)$ is $C^{k-1}$ close to  $h o \phi (s,0)$, hence $h_n o \phi_n$ is $C^{k-1}$ close to  $h o \phi$, for $n$ big enough. Since $\phi_n \overset{C^k}{\to} \phi$, one easily shows that  actually $h_n$ converges to $h$ on a neighborhood of $q$, for the $C^{k-1}$ topology, and then on any compact subset of the connected component of $V$ containing $q$. Now, look at Equation (\ref{Eq diff pour h}), the $C^{k-1}$ convergence of $h_n$ and the $C^k$  convergence of $\bar{K}_n$ imply the $C^{k-1}$ convergence of the first order derivatives of $h_n$, hence the $C^k$ convergence of $h_n$. 
\end{proof}
\begin{fact}\label{Convergence des champs de lumière à partir des champs de Killing}
Let $g_n$ be a sequence of metrics in $\mathcal{L}_K(T)$ converging to $g$ for the $C^r$ topology; $K$ (resp. $K_n$) the Killing field of $g$ (resp. $g_n$) fixed as in the previous fact. Denote by $R$ a maximal ribbon in $\widetilde{T}$, and let $L$ be the null vector field on it defined by $g(L,K)=1$; define $L_n$ (for $n$ big enough)  to be the null vector field on $R$ such that $g_n(K_n,L_n)=1$. The $C^k$ convergence of $K_n$ to $K$ on $R$, with $k \leq r$, leads to the $C^k$ convergence of $L_n$ to $L$.   
\end{fact}
\begin{proof}
Let $\bar{L}$ (resp. $\bar{L}_n$)  be the null vector field on $R$, $\R^+$-collinear to $L$ (resp. $L_n$), such that $S(\bar{L}_n, \bar{L}_n)=S(\bar{L},\bar{L})=1$, where $S$ is a Riemannian metric on $T$. By definition, the sequence of (oriented) foliations associated to $L_n$ converges to that of $L$, so that $\bar{L}_n \overset{C^r}{\to} \bar{L}$.  Write $L=s\bar{L}$ and $L_n=s_n\bar{L}_n$, $s$ and $s_n$ being non-vanishing functions. This yields
$$g(\bar{L},K) = \frac{1}{s} \text{\;\;and\;\;} g_n(\bar{L}_n,K_n) = \frac{1}{s_n}. $$
It follows that if $K$ is $C^k$ close to $K$, then $s_n$ is $C^k$ close to $s$, which finishes the proof.
\end{proof}
\begin{fact}
Let $g_n$ be a sequence of metrics in $\mathcal{L}_K(T)$ converging to $g$ for the $C^r$ topology; take $q \in V \cap U$ and fix $K$ and $K_n$ as in Fact \ref{Fait2: convergence de K_n à partir des feuilletages}. Denote by $R$ a maximal ribbon for the metric $g$, containing $q$, and define $L$ and $L_n$ on $R$ as in the previous fact, then $L_n$ is $C^{r-3}$ close to $L$ on any compact subset of $R$. 
\end{fact}
\begin{proof}
Combining Fact \ref{Fait1: Convergence des feuilletages} and Fact \ref{Fait2: convergence de K_n à partir des feuilletages}, we can say that $L_n$ converges to  $L$ for the $C^{r-3}$ topology on any compact subset of the connected component of $q$ in $V \cap U$.  	Denote by $B$ a set saturated by $K$, where this convergence holds. Now, in a local chart, the geodesics of $g$ are the integral curves of the vector field of $T T(T)$ defined by
$$Z= (Z^1,Z^2) = (y^k, - \sum_{i,j} \Gamma^k_{i,j}(x)y^i y^j).$$
The components of $Z$ involve the first derivatives of $g$, so if $g_n$ converges to $g$ for the $C^r$ topology, then $Z_n$ converges to $Z$ for the $C^{r-1}$ topology. Therefore, there exists an open neighborhood $\mathcal{U}$ of the zero section of $\R \times T(T)$ on which the flows $\Phi_n$ of $Z_n$ are all defined, for all $n$ but a finite number. Furthermore, the sequence $\Phi_n$ converges to $\Phi$ on any compact subset of $\mathcal{U}$ (for the $C^{r-1}$ topology). Let $p \in R$; for an appropriate $t_0 >0$, $\Phi_{t_0}$ is a diffeomorphism from an open set $T \mathcal{V}$ of $T(T)$, where $\mathcal{V}$ is an open set in $B$, into an open set $T \mathcal{W}$, where $\mathcal{W}$ is a neighborhood of $p$. The sequence of diffeomorphisms $\Phi_n(t_0,.)$ converges to $\Phi(t_0,.)$; combining this with the fact that $L_n$ tends to $ L $ on $ \mathcal{V} $, we get $\Phi_n(t_0,L_n) \overset{C^{r-3}}{\longrightarrow} \Phi(t_0,L)$. Now write $\Phi_n (t_0, L_n) = L_n o \Phi_n$ and $\Phi(t_0, L) = L o \Phi$. It follows that $L_n o \Phi_n \overset{C^{r-3}}{\longrightarrow} L o \Phi$, hence $ L_n \overset{C^{r-3}}{\longrightarrow} L $ in the neighborhood of $p$. 
\end{proof}
We go back to the proof of Proposition \ref{Convergence champs Killing}.
\begin{proof}
We start with the following observation. \\
\textbf{Observation:} Suppose $g$ has non-constant curvature $\kappa$, and consider $p \in T$ such that $\alpha:=\kappa(p)$ is a regular value of $\kappa$ and $ g(K,K)(p) \neq 0 $; set $g(K,K)(p)=\eta, \eta = \pm 1$, taking $\lambda_0 K$ if necessary. The set $\kappa^{-1}(\alpha)$ of isolated orbits of $K$ contains the orbit of $K$ containing $p$. Now let $\{\psi^{t}\}$ be a flow of $K$. Since the stages $\psi^{t}$ of the flow are isometries, the exponential map commutes with the flow; we write  
\begin{align}\label{relation commutativité}
exp_{\psi^{t}(p)} o\,\psi^{t}_* = \psi^{t} o \,exp_p,
\end{align}
This geometric ingredient allows, once an orbit of $K$ is given, to get all the orbits of $K$.  Indeed, fix $q \in T$; $p$ and $q$ can be joined by a curve made of (broken) null geodesics, cutting each orbit of $K$ transversally; call this curve $c$. Denote by $\gamma(t)$ the integral curve of $K$ containing $p$; put $\gamma(0)=p$, and $g(\dot{\gamma}(p),\dot{\gamma}(p))=1$. Let $L$ denote the null vector field along $\gamma$, tangent to the null foliation containing $c$, and such that $\langle L,K \rangle=1$. Now (\ref{relation commutativité}) gives:
$$ \psi^{t}(q)= exp_{\gamma(t)}(s L_{\gamma(t)}),$$
where $exp_p (sL(p))=q$, defining the integral curve of $K$ containing $q$. 
\\	

According to Fact \ref{Fait1: Convergence des feuilletages} and Fact \ref{Fait2: convergence de K_n à partir des feuilletages}, $K_n \overset{C^{r-3}}{\longrightarrow} K$ on every compact subset of the connected component of $p$ in $V \cap U$. Denote by $B$ an open set containing $p$ where this convergence holds. Furthermore, in every ribbon containing $p$, the sequence $L_n$ of null vector fields defined in a ribbon by $g_n(L_n,K_n)=1$, for $n$ large enough, converges to the null vector field $L$ defined in that ribbon by $g(L,K)=1$. \\
Now, we denote by $Exp$ (resp. $Exp_n$) the exponential map of $g$ (resp. $g_n$). Let $q$ be a point in a ribbon containing $p$, and let $t_0 >0$ such that $q^{'} = Exp(t_0 L(q)) \in B$. There exists a neighborhood $\mathcal{W}$ of the zero-section of a subset of $T(T)$, such that $\mathcal{W}=\{t v, t \in [0,t_0], v \in \mathcal{U}\}$, where $\mathcal{U}$ is a neighborhood of $L(q)$ in $T(T)$, in which the exponential maps of the metrics $g_n$ are defined for all $n$ but a finite number. Furthermore, $Exp_n$ converges $C^{k-1}$ to $Exp$ on every compact subset of $\mathcal{W}$. Now, according to the previous observation, if $z \in \pi(\mathcal{U})$, where $\pi: T(T) \to T$ is the natural projection, the integral curve of $K$ (resp. $K_n$) containing $z$ is given by $\psi^{t}$ (resp. $\psi_n^{t}$), such that:  
$$\psi^{t}(z)= Exp_{\psi^{t}(\phi(z))}(-t_0 L(\psi^{t} o \phi(z)))), \text{\;\;where\;\;} \phi(z)=Exp_z(t_0L(z)) \in B.$$
Since $L_n$ converges $C^{r-3}$ to $L$, and the flows of $K_n$ converge $C^{r-3}$ to the flow of $K$ in $B$, then  $\psi_n^{t}$ converges $C^{r-3}$ to $\psi^{t}$, leading to the $C^{r-4}$ convergence of $K_n$ to $K$ on $\pi(\mathcal{U})$. The result on the torus follows from the fact that any point on it can be brought to a point in $B$ by a broken null geodesic everywhere transverse to $K$. 
\end{proof}

\begin{corollary}\label{corollaire convergence de champs K_n}
Let $(T,K)$ be a torus in $\mathcal{L}^*_K(T)$; the number of bands, their signs and types are preserved in a neighborhood of the torus in $\mathcal{L}_K(T)$. 
\end{corollary}
\begin{proof}
For the number of bands and their signs, this is an easy consequence of the previous proposition. For the type  preservation, recall that the null-leaves of $ K $ on a band of type II belong to different null-foliations, while on a band of type I, they belong to the same null-foliation. Since we cannot have a type III band in the torus, the result follows. 
\end{proof}

\begin{corollary}\label{Convergence des cartes locales}
Let $g \in \mathcal{L}_K (T)$ a non-flat metric with Killing field $K$, and let $g_n \in \mathcal{L}_K (T)$ be a sequence of metrics, with Killing field $K_n$, converging to $g$ for the $C^r$ topology, $r \geq 4$. Let $R$ be a maximal ribbon in $\widetilde{T}$ for the metric $g$, and fix $p \in R$.  Define $\phi: R \to I \times \R, \phi(q) =(x(q),y(q))$, with $\phi(p)=(0,0)$; define in the same way $\phi_n$ on $R$ satisfying $\phi_n(p)=(0,0)$. Then, $\phi_n$ converges $C^{r-4}$ to $\phi$ on every compact subset of $R$.
\end{corollary}
\begin{proof}
Proposition \ref{Convergence champs Killing} implies the $C^{r-4}$ convergence of the Killing fields, hence that of the transverse coordinates, using Lemma \ref{Convergence des rubans}. Now, define a volume form $\nu$ on $R$ by setting $\nu(K,L)=1$, where $L$ is the null vector field on $R$ satisfying $g(L,K)=1$. The coordinate $y \in C^{\infty}(R,\R)$ can  be defined by $i_{\widetilde{L}} \nu = dy$ and $y(p)=0$. In the same way, consider for $n$ sufficiently large a volume form $\nu_n$ on $R$ defined by $\nu_n(K_n,L_n)=1$, where $L_n$ is the null vector field on $R$ satisfying $g_n(L_n,K_n)=1$. The coordinate $y_n \in C^{\infty}(R,\R)$ is given by $i_{L_n} \nu_n =d y_n$, $y_n(p)=0$. The convergence of $\phi_n$ to $\phi$ follows from Proposition \ref{Convergence champs Killing} and Fact \ref{Convergence des champs de lumière à partir des champs de Killing}.
\end{proof}
\begin{remark}
The application $\phi$ (resp. $\phi_n$) defined above is a diffeomorphism on $R$ (resp. on every saturated subset of $R$). Take an open set $J$ of $I$, contained in $I_n$ for all $n$ but a finite number. The pullback of the metrics by these applications read $(\phi^{-1})_* g = 2dx dy + f(x) dy^2$ and $(\phi^{-1}_n)_* g_n = 2dxdy + f_n(x) dy^2$, $(x,y) \in I \times \R$, whose extensions to $\R^2$ that give the maximal ribbons in $E^u_f$ and $E^u_{f_n}$ respectively, are proved to be convergent in Lemma \ref{Convergence des rubans}.  
\end{remark}
\begin{lemma}\label{Convergence des cartes des selles}
Let $(U,K)$, $I=I \times \R$ be a Lorentzian domino where the unique null orbit of $K$ (represented by $x=0$) is incomplete. Denote by $g$ the metric on $U$ and let $(g_n,K_n)$ be a sequence of metrics on $U$ such that $(g_n,K_n) \overset{C^r}{\to} (g,K), r \geq 2$. Then, there exists a neighborhood $J$ of $0$ in which the coordinate neighborhoods for $g_n$ given by Equation (\ref{Selle: expression explicite de la métrique}) converges $C^{r-2}$ to that of $g$. 
\end{lemma}
\begin{proof}
The neighborhood $J$ is given by Lemma \ref{Convergence des selles}, and the convergence follows from the previous lemma, and Equations (5) and (8), \cite{10}, that define these coordinates.
\end{proof}
\subsection{A continuous function controlling conjugate points}\label{Section: A continuous function controlling conjugate points}
Define $$\Omega=\{(g,p) \in \mathcal{L}_K^*(T) \times T,\; g(K,K)(p) \neq 0 \text{\;\;and\;\;} \nabla_K K (p) \neq 0\}.$$
For $g \in \mathcal{L}_K^*(T)$, the points $p \in T$ such that $(g,p) \in \Omega$ are the points of $T$ which are neither on a null orbit of $K$ nor on a critical orbit. Take $(g,p) \in \Omega$ and consider the geodesic $\gamma_p$ in $T$ tangent to $K$ at $p$. Set $t=0$ at $p$ and define $c(t_0)+c(t_1)$ as in the previous section. 
Define a function $\mathcal{Z}$ on the open set $\Omega$ as follows: 
\begin{align*}
\mathcal{Z}: \Omega &\longrightarrow \R \\
\mathcal{Z}(g,p) & = \frac{c(t_0)+c(t_1)}{s(t_0)}
\end{align*}
where $t_0$, $t_1$ and $c$ are exactly as in the previous section.
\begin{lemma}
A metric $g$ in $\mathcal{L}^*_K(T)$ is without conjugate points if and only if $\mathcal{Z}(g,p) \geq 0$, $\forall p \in T$ such that $(g,p) \in \Omega$. 
\end{lemma}
\begin{proof}
This follows from Lemma \ref{orbites critiques SPC} and Proposition \ref{SPC}.	
\end{proof}
\begin{proposition}
If the space of Lorentzian metrics on $T$ is equipped with the $C^2$ topology, then $\mathcal{Z}$ is continuous on $\Omega$. 
\end{proposition}
\begin{proof}
Let $(g_n,p_n) \rightarrow  (g,p) $ in $ \Omega $, where $g_n$ converges to $g$ for the $C^2$ topology. In $\R^2$ equipped with the metric $2dxdy + f(x) dy^2$ (resp. $2dxdy + f_n(x) dy^2$), let $\gamma$ (resp. $\gamma_n$) be the geodesic tangent to $\phi_* K$ (resp. $\phi_{n_*} K_n$) at $\phi(p)$ (resp. $\phi_n(p_n)$), where $\phi, \phi_n, \,n \in \N$, are the local charts defined in Corollary \ref{Convergence des cartes locales}, satisfying $\phi(p)=\phi_n(p)=0$. The geodesic sequence $ \gamma_{n}$ converges uniformly to $ \gamma$ on any compact set of $\R^2$. 
\\
Let us rewrite the Jacobi equations along $ \gamma_n $ and $ \gamma $:
\begin{align}\label{éq}
        u''(t)+\epsilon \kappa_n(t) u(t) &=0 \\
       u''(t)+ \epsilon \kappa(t) u(t) &=0
\end{align}
$\kappa_n$ converges uniformly to $\kappa$ on any compact set. The pairs of solutions $ (s, c) $ and $ (s_n, c_n) $ are given by the same initial conditions so $s_n$ and $ c_n $ converge uniformly to $s$ and $ c $ respectively, on any compact set containing $\gamma$. \\
Lemma \ref{Behavior of geodesics} states that $\gamma$ restricted to the subset of $\R^2$ contained in $\widetilde{T}$ lies in a domino. This holds also for $\gamma_n$, for $n$ large enough. Denote by $x_0$ (resp. $x_0^n$) the zero of $f$ (resp. $f_n$) in it, and set $x=z_0$ (resp. $x=z_0^n$) at $p$ (resp. $p_n$). We have by use of Fact \ref{Fait sur la convergence des zéros uniques de fonctions sur un compact} that $x_0^n$ converges to $x_0$. Write $$t_0 = \int_{z_0}^{x_0} \frac {1} {\sqrt {C^2 - \epsilon f(x)}} dx.$$
This yields $t_0^n \to t_0$, hence the convergence of $c_n(t_0^n)$ to $c(t_0)$.  
To show that $ c_n (t_1^n) $ tends to $ c (t_1) $, we also show that $ t_1^n $ tends to $ t_1 $. Since a geodesic crosses a type I band between $ t_0 $ and $ t_1 $, we can write $$ t_1 = t_0 + \int_{x_0}^{x_1} \frac {1} {\sqrt {C^2 - \epsilon f(x)}} dx. $$
This yields $ t_1^n \to t_1 $ in the same way, finishing the proof. 
\end{proof}
Now assume that $T \in S\mathcal{L}^*_K(T)$, i.e. the function $f$ induced by the norm of $K$ satisfies the property
 $$f^{'}(x_n) + f^{'}(x_{n+1}) = 0,\;\forall n \in \Z, $$ 
 where $(x_n)_{n \in \Z}$ is the sequence of zeros of $f$, taken in an increasing order. Then, using $c s^{'}- c^{'} s=1$, one can check that the $\mathcal{Z}$ function can be written in the following way:
\begin{align*}
\mathcal{Z}: \Omega &\longrightarrow \R \\
\mathcal{Z}(g,p) & = \frac{c^{'}(t_0)-c^{'}(t_1)}{s^{'}(t_0)}
\end{align*}
\begin{proposition}
The functions $\mathcal{Z}$ can be extended to the set 	
\begin{align*}
\widetilde{\Omega}=&\{(g,p) \in S\mathcal{L}_K^*(T) \times T,\; g(K,K)(p) \geq 0 \textrm{\;\;and\;\;} \nabla_K K (p) \neq 0 \}\\ \amalg
& \{(g,p) \in S\mathcal{L}_K^*(T)\times T,\; g(K,K)(p) \leq 0 \textrm{\;\;and\;\;} \nabla_K K (p) \neq 0\},
\end{align*}
into a continuous function.
\end{proposition}
\begin{proof}
Take $(g,p) \in S\mathcal{L}_K^*(T) \times T$ such that $g(K,K)(p)=0$. The null orbit of $K$ containing $p$ splits the domino containing it into two connected components (interior of bands), one timelike band and one spacelike band. Take $(p_n)_n$ a sequence of points in the torus converging to $p$, and contained in the timelike band. The sequence $(\gamma_n)_n$  of timelike geodesics tangent to $K$ at $p_n$, $\forall n$, converges to the null orbit of $K$ containing $p$. The function $\mathcal{Z}(g)$ being constant along the leaves of the Killing field for each metric $g$, we don't change the behavior of $\mathcal{Z}$ if we move the $\gamma_n$'s by the flow of $K$ in such a way that the sequence converges to a timelike geodesic $\gamma_{\infty}$ orthogonal to $K$ and crossing the spacelike neighboring band. We can do this by letting $p_n$ go to the saddle point of the null-orbit of $K$ containing $p$, along a leaf orthogonal to $K$ contained in $B$. This way, we see that $\mathcal{Z}(g,p_n)$ converges to $-c_{\infty}^{'}(2\omega)$, where $s_{\infty}$ and $c_{\infty}$ are the solutions of the Jacobi equation on $\gamma_{\infty}$, with initial conditions at the saddle point. This limit does not depend on the choice of $\gamma_{\infty}$; denote it by $\mathcal{Z}(g,p)^{-}$. In the same way, if $\gamma_n$'s are spacelike, we get another limit using this time a timelike geodesic orthogonal to $K$ and passing through the same saddle point. Now, fix a metric in $S\mathcal{L}_K^*(T)$ and denote by $(B_i)_{i \in I}$ the bands of the torus for the metric $g$; the boundary of a band $B_i$ is made of two null orbits of $K$, call them $\partial_1 B$ and $\partial_2 B$. Define for $p \in \amalg_{i \in I} B_i$:
$$
\widetilde{\mathcal{Z}}(g,p)=\left\{
\begin{array}{ll} 
\mathcal{Z}(g,p) \text{\;\;\;\; if \;} (g,p) \in \mathring{B}_i,\\
\mathcal{Z}(g,p)_1^{\epsilon} \text{\;\;\;\;if\;} (g,p) \in \partial_1 B,\\
\mathcal{Z}(g,p)_2^{\epsilon} \text{\;\;\;\;if\;} (g,p) \in \partial_2 B,
\end{array}
\right.
$$
for all $i \in I$, where $\epsilon={\pm 1}$, depending on the type of the band.  \\
This gives a well defined function $\widetilde{\mathcal{Z}}$ on the set 
\begin{align*}
\widetilde{\Omega}=&\{(g,p) \in S\mathcal{L}_K^*(T) \times T,\; g(K,K)(p) \geq 0 \textrm{\;\;and\;\;} \nabla_K K (p) \neq 0 \}\\ \amalg
& \{(g,p) \in S\mathcal{L}_K^*(T)\times T,\; g(K,K)(p) \leq 0 \textrm{\;\;and\;\;} \nabla_K K (p) \neq 0\},
\end{align*}	

Now, to prove that $\widetilde{\mathcal{Z}}$ is continuous, take $(g,p)  \in S\mathcal{L}_K^*(T) \times T$ such that $g(K,K)(p)=0$, and consider a sequence of metrics $g_n \in S\mathcal{L}_K^*(T)$ converging to $g$. Lemma \ref{Convergence des cartes des selles} and Lemma \ref{Convergence des selles} allow us to consider a coordinate neighborhood centered in $0$ and given by Equation (\ref{Selle: expression explicite de la métrique}), in which the saddle points of $g_n$ and that of $g$ are represented by the origin. Take a sequence $\gamma_n$ of timelike geodesics converging to the null orbit of $K$ containing $p$.  Denote by $p_n$ the points where $\gamma_n$ is tangent to $K_n$. One can find a sequence of points $q_n$ such that $\widetilde{\mathcal{Z}}(g_n,q_n)= \widetilde{\mathcal{Z}}(g_n,p_n)$, for all $n \in \N$, and the sequence of geodesics $\sigma_n$ tangent to  $K_n$ at $q_n$ converges to a timelike geodesic $\sigma_{\infty}$ orthogonal to $K$. We can achieve this in the following way: there exists a sequence $(t_n)_n$ such that the sequence of points $F^{t_n}(p_n)$, where $F^t$ is the flow of $K$,  converges to the origin. We use the convergence of the flows and define $q_n := F_n^{t_n}(p_n)$, where $F_n^t$ is the flow of $K_n$  (we choose the points $q_n$ to vary on a leaf orthogonal to $K$). This way, we obtain $c_n^{'}(t_0^n) \to c_{\infty}^{'}(0)=0$ and $s_n^{'}(t_0^n) \to s_{\infty}^{'}(0)=1$, where $s_{\infty}$ and $c_{\infty}$ are the solutions of the Jacobi equation on $\gamma_{\infty}$, with initial conditions at the origin. 	\\
Now, we want to prove that $c_n^{'}(t_1^n)$ converges to $c_{\infty}^{'}(2 \omega_{\infty})$. Take an interval $I$ in which $\sigma_{n}$ converges to $\sigma_{\infty}$. Take $\tau_0 \in I$; $\sigma_n$ and $\sigma_{\infty}$ may be seen as solutions of the differential equation which gives locally the geodesics of $g_n$ and $g$, with initial conditions given by $\sigma_{n}^{'}(\tau_0)$ and $\sigma_{\infty}^{'}(\tau_0)$ respectively. Let us consider a coordinate neighborhood $U$, given by Lemma \ref{Convergence des selles}, containing the saddle point that belongs to the null orbit of the type I neighboring band containing $\sigma_{\infty}^{'}(\tau_0)$. Choosing $\tau_0$ so that $\sigma_{\infty}(\tau_0), \sigma_{n}(\tau_0) \in U$ for $n$ large enough, we can assert that $\sigma_{n}$ (extended to $U$) converges to $\sigma_{\infty}$ in $U$. The result follows.     	
\end{proof}
From now on, the function $\widetilde{\mathcal{Z}}$ will be simply denoted by $\mathcal{Z}$.
\subsection{Digression: Jacobi equations all of whose solutions are periodic}\label{Section: Digression: Jacobi equations all of whose solutions are periodic}

In what follows, $g$ belongs to the family given in Theorem \ref{famille SPC}.
\begin{lemma}
	If a metric $g \in \mathcal{L}_K(T)$ belongs to the family given in Theorem \ref{famille SPC}, then $\mathcal{Z}$ vanishes on $(g,p), p \in T$, if and only if the Jacobi equation corresponding to $\gamma_g$, the $g$-geodesic tangent to $K$ at $p$, has only periodic solutions. 
\end{lemma}
\begin{proof}
For such a metric, $\mathcal{Z}$ is defined for $p \in T$ such that $\gamma_g$ is an invariant geodesic, and the additional isometries of $g$ given by condition (iv) in Theorem \ref{famille SPC} imply that $\kappa(t)$ is $2 \omega$ periodic and that $\beta$ is $2\omega$ anti-periodic. In this case, formula (\ref{Relation sur c, symétries ajoutées}) in Lemma \ref{Lemme beta distance minimale} holds, so that $\mathcal{Z}$ vanishes if and only if $c(\omega)=0$, which is equivalent to saying that $c$ is $2 \omega$ anti-periodic (hence $4 \omega$ periodic).
\end{proof}

\begin{remark}\label{Hill equations remark}
An equation $$u^{''}+r(t)u=0, \;t\in \R,$$
where the coefficient $r(t)$ is $T$-periodic is called a Hill equation. When all the solutions are assumed to be $T$-antiperiodic, with one zero on $[0,T[$, the following inequality holds (see \cite{5}, Appendix B p. 230)\\
$$T \int_{0}^{T} r(t) dt \leq \pi^2,$$
with equality only for constant $r(t)$.
\end{remark}
Let $p, r : [a,b] \to \R$, where $p(x) >0$, and consider the equation
\begin{align}\label{Sturm-Liouville Equation}
(p y')^{'} + \lambda r y =0, a \leq x \leq b,
\end{align}
with boundary conditions
\begin{align}\label{Sturm-Liouville boundary conditions}
y(a)=y(b)=0.
\end{align}
According to \cite{1} (p. 288), there are two infinite sequences of parameter values $0 \leq \lambda_1 \leq \lambda_2 ...$, $0 \geq \lambda_{-1} \geq \lambda_{-2} ...$, each one of which has $+ \infty$ and $- \infty$ for its only point of accumulation, and for each parameter $\lambda_m$ (resp. $\lambda_{-m}$), a solution $y_m$  (resp. $y_{-m}$) satisfying (\ref{Sturm-Liouville boundary conditions}) exists. The number of zeros of $y_m$ in $[a,b]$ is $m+1$.

Let $\gamma$ be an invariant geodesic for the metric $g$, and consider the differential system
\begin{align}\label{Sturm-Liouville, Jacobi equation 1}
\left\{
\begin{array}{ll}
&u^{''} + \lambda \epsilon \kappa u = 0, \;\; t \in [0, 2 \omega] \\
&u(0)=u(2 \omega)=0,
\end{array}
\right.
\end{align}
where $\kappa$, the sectional curvature along $\gamma_g$, is (see paragraph \ref{les exemples de tores SPC})  symmetric with respect to $t=0$ and $t= \omega$. Denote by $\lambda_1$ the least positive eigenvalue of the differential system (\ref{Sturm-Liouville, Jacobi equation 1}), which exists by the statement above, and let $y_1$ be the corresponding eigenfunction. Since $\kappa$ is symmetric with respect to $t= \omega$, an easy computation gives either $y_1^{'}(\omega)=0$ or $y_1(\omega)=0$. Furthermore, $y_1$ does not vanish on $]0,2\omega[$, so actually $y_1^{'}(\omega)=0$. It follows that $y_1$ is a solution of the new system 
\begin{align}\label{Sturm-Liouville, Jacobi equation 2}
\left\{
\begin{array}{ll}
&u^{''} + \lambda \epsilon \kappa u = 0, \;\; t \in [0, \omega] \\
&u(0)=u^{'}( \omega)=0.
\end{array}
\right.
\end{align}
Now, if $\lambda$ is an eigenvalue of the system (\ref{Sturm-Liouville, Jacobi equation 2}), with corresponding eigenfunction $u$ defined on $[0,\omega]$, then, by reflecting $u$ about the line $t=\omega$, one gets a solution of the differential system  (\ref{Sturm-Liouville, Jacobi equation 1}), for $\kappa$ is symmetric with respect to $\omega$.\\
This proves that $\lambda_1$ is also the least positive eigenvalue of the differential system (\ref{Sturm-Liouville, Jacobi equation 2}).

The following fact follows from the proof of Lemma 3, \cite{3}.
\begin{fact}\label{Lambda_1=1}
Suppose that the coefficient $r$ in Equation (\ref{Sturm-Liouville Equation}) is symmetric with respect to $t_0= \frac{a+b}{2}$. If there exists a solution $y_0$ of $(p y')^{'}+ r y =0$ satisfying $y_0(a)=y_0(b)=0$, symmetric with respect to $t_0$, with $y_0$ not vanishing on $]0,2\omega[$, then the least positive eigenvalue of the system (\ref{Sturm-Liouville Equation})-(\ref{Sturm-Liouville boundary conditions}) is $\lambda_1=1$.
\end{fact}
\begin{lemma}\label{lemme inégalité 0, contrainte sur kappa}
	When $g$ belongs to the family given in Theorem \ref{famille SPC}, the following inequality holds
	$$ \int_{0}^{2 \omega} \epsilon \kappa(t) dt \leq \frac{\pi^2}{2 \omega}.$$
	for all invariant geodesics. 
\end{lemma}
\begin{proof}
Let $\lambda_1$ be the least positive eigenvalue of the system
\begin{align}\label{Système Jacobi Equation 1, local}
\left\{
\begin{array}{ll}
&u^{''}+ \lambda \epsilon \kappa u =0, \;\;\;t \in I=[0, \omega]  \\
&u(0)=u^{'}( \omega)=0.
\end{array}
\right.
\end{align} 
We have the following inequality from Theorem 6, \cite{4}
$$\lambda_1 \leq \frac{(\pi/2)^2}{d^2 \omega^2}, \text{\;\; where \;\;} d^2=\min_{t \in I} \frac{\int_{t}^{ \omega} \epsilon \kappa(t)dt}{ \omega -t}.$$
The function $\beta$ is a solution of the Jacobi equation that satisfies $$\beta(0)=\beta^{'}(\omega)=0, \;\;\; \beta^{'}>0 \text{\;on\;} [0,\omega[,$$
This implies using Fact \ref{Lambda_1=1} above that the least positive eigenvalue $\lambda_1$ of the system (\ref{Système Jacobi Equation 1, local}) is $\lambda_1=1$. 
Now set $$h(t)=\frac{1}{ \omega -t}\int_{t}^{ \omega} \epsilon \kappa(t)dt, \;\; t \in I$$
We have $$h^{'}(t) = \frac{1}{\omega-t}\Big{[}\frac{1}{\omega-t}\int_{t}^{\omega}\epsilon \kappa(t)dt-\epsilon \kappa(t)\Big{]}.$$
Using the mean value theorem, we show the existence of $\tau \in ]t,\omega[$, such that $\frac{1}{\omega-t}\int_{t}^{\omega}\epsilon \kappa(t)dt = \epsilon \kappa(\tau)$. This gives 
$$h^{'}(t)=\frac{1}{\omega-t}(\epsilon \kappa(\tau) - \epsilon \kappa(t)), \; \tau \in ]t,\omega[.$$
Condition (iii) in Theorem \ref{famille SPC} implies that $\epsilon k(t)$ is an increasing function on $[0,\omega]$, so that 
$$\min_{t\in I} \frac{\int_{t}^{\omega}\epsilon \kappa(t)dt}{\omega-t} = \frac{1}{\omega} \int_{0}^{\omega} \epsilon \kappa(t) dt. $$
\end{proof}
\begin{lemma}\label{lemme inégalité (1)}
	Consider a metric in the family given in Theorem \ref{famille SPC}; suppose that the Jacobi equation of a geodesic $\gamma$ has only periodic solutions, then 
	\begin{align}\label{première inégalité}
	\sup \big{\{} \frac{1}{\omega} \int_{\omega}^{2\omega} \epsilon \kappa(t) dt,\;-\epsilon \kappa(0) \big{\}} \geq \frac{\pi^2}{4\omega^2}.
	\end{align} 
	It follows that either $- \epsilon \kappa (0) \geq \frac{\pi^2}{4 \omega^2}$ or $\kappa$ is constant along $\gamma$.
\end{lemma}
\begin{proof}
Let $\lambda_1$ be the least positive eigenvalue of the system
\begin{align}
\left\{
\begin{array}{ll}
&u^{''}+ \lambda \epsilon \kappa u =0, \;\;\;t \in J=[0, \omega]  \\
&u(\omega)=u^{'}(2 \omega)=0.
\end{array}
\right.
\end{align} 
The following inequality follows from Theorem 5, \cite{4}:
$$\lambda_1 \geq \frac{(\pi/2)^2}{D^2 \omega^2}, \text{\;\; where \;\;} \max_{t \in I} \frac{|\int_{t}^{2 \omega} \epsilon \kappa(t)dt|}{2 \omega -t}=D^2.$$
Suppose that the Jacobi equation has only periodic solutions, i.e. $c(\omega)=0$. In this case,  we have $c(\omega)=c(3 \omega)=0$, and $c(4\omega-t)=c(t)$. Thus, $c$ and the coefficient $\kappa$ are symmetric with respect to $t=2\omega$, with $c$ not vanishing on $]\omega,3 \omega[$. It follows that $\lambda_1=1$, by use of Fact \ref{Lambda_1=1} again.\\ Now	set as in the proof of the previous lemma $$h(t)=\frac{1}{2 \omega -t}\Big{|}\int_{t}^{2 \omega} \epsilon \kappa(t)dt\Big{|}, \;\; t \in I$$
Let $b \in I$ the smallest real such that $\int_{b}^{2\omega} \epsilon\kappa(t) dt=0$. 
Doing the same computation as before, we show the existence of $\tau \in ]t,2\omega[$ such that 
\begin{align*}
h^{'}(t)=\Big{\{}\begin{array}{cc}
\frac{1}{2\omega-t}\Big{(}\epsilon \kappa(\tau)-\epsilon \kappa(t)\Big{)}, & t \in [\omega,b]\\
\frac{1}{2\omega-t}\Big{(}-\epsilon \kappa(\tau)+\epsilon \kappa(t)\Big{)}, & t \in [b,2 \omega]
\end{array}
\end{align*}
Condition (iii) in Theorem \ref{famille SPC} implies that  $ \epsilon \kappa(t)$ is a decreasing function on $[\omega,2\omega]$, so that 
$$\max_{t \in I} \frac{ |\int_{t}^{2\omega} \epsilon \kappa(t) dt|}{2\omega-t} =\sup \Big{\{} \frac{1}{\omega} \int_{\omega}^{2\omega} \epsilon \kappa(t) dt, \; - \epsilon \kappa(0) \Big{\}}.$$
The above inequality reads 
$$\sup \big{\{} \frac{1}{\omega} \int_{\omega}^{2\omega} \epsilon \kappa(t) dt,\;-\epsilon \kappa(0) \big{\}} \geq \frac{\pi^2}{4\omega^2}.$$
The last assertion is an easy consequence of Lemma \ref{lemme inégalité 0, contrainte sur kappa} and Remark \ref{Hill equations remark}. 
\end{proof}
\begin{corollary}
	Suppose that the curvature has a constant sign on an invariant geodesic $\gamma$ without being constant, then the Jacobi equation admits a non-periodic solution. 
\end{corollary}
\begin{proof}
A geodesic of type $\epsilon$ cuts the critical orbit of $K$ of type $-\epsilon$, corresponding to $t=\omega$, on which the curvature is either zero or has sign $\epsilon$; since $\kappa$ doesn't change sign on $\gamma$, $- \epsilon \kappa(0) \leq 0$. Therefore, if the Jacobi equation has only periodic solutions, $\kappa$ is constant on $\gamma$ by use of the previous lemma. 	
\end{proof}
\begin{lemma}\label{lemme inégalité (2)}
	Suppose that the curvature vanishes exactly twice on the smallest period of $f$, and let $\gamma$ be a geodesic on which $\kappa(t) := \kappa o \gamma(t)$ changes sign. If the Jacobi equation associated to $\gamma$ admits only periodic solutions, then
	\begin{align}
	(\omega-\tau)^2 \epsilon \kappa(\omega) \geq \frac{\pi^2}{4},
	\end{align}
	where $0<\tau<\omega$ is the smallest positive real such that $\kappa(\tau)=0$ on $\gamma$. 
\end{lemma}
\begin{proof}
	By assumption, $c(\omega)=0$. Let $b>0$ be the smallest positive real such that $c^{'}(b)=0$. If we assume $\kappa$ to vanish twice in the period of $f$, then the zeros of $\kappa$ are on both sides of a critical orbit of $K$ corresponding to an extremum of $f$. Indeed, since $\kappa(t)$ is symmetric with respect to $t=0$ and $t= \omega$, either both zeros of it are in the same band or both of them are on the critical orbits of $K$ where $f^{'}$ changes sign. The latter implies that $\kappa$ is everywhere positive or everywhere negative on the torus, so it cannot actually happen, unless the torus is flat. In addition, $\kappa$ changes sign while vanishing.  
 It follows that $\kappa$ has sign $-\epsilon$ on $[0,\tau[$ and $\epsilon$ on $]\tau,\omega]$, so that $c$ is convex on $[0,\tau]$ and concave on $[\tau,\omega]$, hence $\tau < b < \omega$. The curvature does not vanish on $[b,\omega]$; on this interval  it is easy to check that the differential equation 
	$$\big{(}\frac{1}{\kappa} y^{'}\big{)}^{'} + \epsilon y =0$$
	is satisfied by $u^{'}$, where $u$ is a solution of the Jacobi equation. Let $\lambda_1$ be the least positive eigenvalue of the system
	\begin{align*}
\big{(}\frac{1}{\kappa} y^{'}\big{)}^{'} +\lambda \epsilon y &=0, \;\;\;\;\;t \in I=[b, \omega], \\
	y(b)&=y^{'}(\omega)=0.
	\end{align*} 
	The assumption on $c$ leads to $c^{''}(\omega)=0$, so we have  $c^{''}(\omega)=c^{'}(b)=0$, and  $c^{''}$ does not vanish on $[b,\omega[$; it follows that $\lambda_1=1$, as in the proof of the previous lemma.\\
	We have (see Theorem 5, \cite{4}) 
	\begin{align}\label{inégalité ici}
\max_{t \in I} \frac{\omega-t}{\int_{t}^{\omega} \epsilon \kappa(t) dt} \geq \frac{\pi^2}{4(\int_{b}^{\omega} \epsilon \kappa(t) dt)^2}.
	\end{align}
	Now set $$h(t)=\frac{\omega -t}{\int_{t}^{\omega} \epsilon \kappa(t) dt}, \;\; t \in [b,\omega],$$
	The same argument as in the previous lemmas shows the existence of $\theta \in ]t,\omega[$, such that
	$$h^{'}(t)=(\omega -t)(-\epsilon \kappa(\theta) + \epsilon \kappa(t)) \frac{1}{(\int_{t}^{\omega} \epsilon \kappa(t) dt)^2},$$
	so that $h$ is a decreasing function on $[b,\omega]$, and (\ref{inégalité ici}) reads 
	$$(\omega-b) \int_{b}^{\omega} \epsilon \kappa(t) dt \geq \frac{\pi^2}{4}.$$
	Combining this with the fact that $\epsilon \kappa(t)$ reaches its maximum at $t=\omega$, and that $b>\tau$, we get the desired inequality:
	$$(\omega - \tau)^2 \epsilon \kappa(\omega) \geq \frac{\pi^2}{4}.$$	
\end{proof}
\subsection{Application: Invariant geodesics with non-periodic Jacobi vector field}\label{Application: Invariant geodesics with non-periodic Jacobi vector field}
In what follows, $f \in C^{\infty}(\R,\R)$ belongs to the family given in Theorem \ref{famille SPC}. 
\begin{proposition}\label{proposition inégalité (1)}
	Suppose $\kappa$ vanishes twice on the smallest period of $f$. If there  exists an invariant geodesic $\gamma_0$ on which $\kappa$ changes sign, such that 
	\begin{align}\label{inégalité (1)}
	-\epsilon f^{''}(x_0) \Big{(}\int_{x_0}^{x_1} \frac{dx}{\sqrt{M_{\epsilon} -\epsilon f(x)}}\Big{)}^2 < 2 \pi^2,
	\end{align}
	where  $x_0$, $x_1$ are the coordinates of two consecutive points where $\gamma_0$ is tangent to $K$, and $M_{\epsilon}=\sup \epsilon f(x)$, then, either $\kappa$ is constant on the band where $\epsilon f$ is negative, or for all the geodesics such that $C^2 \leq C_0^2=\epsilon f(x_0)$, where $C_0$ is the Clairaut constant of $\gamma_0$, the Jacobi equation admits a non-periodic solution.
\end{proposition}
\begin{proof}
Let $z_0=0$ and $d$ be two consecutive zeros of $f$, such that $z_0$ and $d$ border the part of $f$ of sign $- \epsilon$. Denote by $x_{cr}^{\epsilon}$ the coordinate of a zero of $f^{'}$, such that either $x_0 \in  [x_{cr}^{\epsilon},0]$ or $x_0 \in [d,x_{cr}^{\epsilon}]$; for simplicity, assume $x_0 \in [x_{cr}^{\epsilon},0]$. We have $x_1=d-x_0$. 
Define $$h(x)=-\epsilon \frac{f^{''}(x)}{2} -\frac{\pi^2}{\Big{(}\int_{x}^{d-x}\frac{dt}{\sqrt{M_{\epsilon}-\epsilon f(t)}}\Big{)}^2},$$
where $x\in J=[x_0,0]$; $h$ is a derivable function since for $x \in J$, $M_{\epsilon}-\epsilon f(t) > 0$ for every $t \in [x,d-x]$.\\
A simple computation gives 
$$h^{'}(x)=-\epsilon \frac{f^{(3)}(x)}{2} - \frac{2\pi^2}{\sqrt{M_{\epsilon}-\epsilon f(x)}\Big{(}\int_{x}^{d-x} \frac{dt}{\sqrt{M_{\epsilon}-\epsilon f(t)}}\Big{)}^3}.$$
Multiply both sides by $f^{'}(x)$; this gives
$$f^{'}(x)h^{'}(x)=-\epsilon \frac{f^{(3)}(x)f^{'}(x)}{2} - D^2 f^{'}(x),$$ where $D^2=\frac{2\pi^2}{\sqrt{M_{\epsilon}-\epsilon f(x)}\Big{(}\int_{x}^{d-x} \frac{dt}{\sqrt{M_{\epsilon}-\epsilon f(t)}}\Big{)}^3}.$\\
By assumption, we have $f^{'}f^{(3)} \leq 0$; furthermore, $f^{'}$ has sign $-\epsilon$ on $J$; combining these two facts, we deduce that $h$ is a decreasing function on $J$. Now, for $C^2 \leq C_0^2$ \;i.e. \;$x \in [x_0,0]$, we have 
$$h(x_0) \geq -\epsilon \frac{f^{''}(x)}{2} -\frac{\pi^2}{(\int_{x}^{d-x} \frac{dt}{\sqrt{C_0^2-\epsilon f(t)}})^2},$$
so if we suppose that (\ref{inégalité (1)}) is true, then $$-\epsilon \frac{f^{''}(x)}{2} -\frac{\pi^2}{\Big{(}\int_{x}^{d-x} \frac{dt}{\sqrt{C^2-\epsilon f(t)}}\Big{)}^2} < 0, \;\; \forall x \in J.$$ 
Now, recall that $\kappa(t)=f^{''}(x(\gamma(t)))/2$, and $t(d-x)-t(x)=\int_{x}^{d-x} \frac{1}{\sqrt{C^2-\epsilon f(x)}} dx$. So the latter inequality is equivalent to $- \epsilon \kappa(0) < \frac{\pi^2}{4 \omega^2}$ on $\gamma_C$, which ends the proof by Lemma \ref{lemme inégalité (1)}.
\end{proof}
\begin{proposition}\label{proposition inégalité (2)}
	Suppose $\kappa$ vanishes twice on the smallest period of $f$. Assume there  exists an invariant geodesic $\gamma_0$ on which $\kappa$ changes sign, such that 
	\begin{align}\label{inégalité (2)}
	\epsilon f^{''}(x_{cr}^{-\epsilon}) \Big{(}\int_{\zeta_0}^{\zeta_1} \frac{dx}{\sqrt{\epsilon f(x_0) -\epsilon f(x)}}\Big{)}^2 < 2 \pi^2,
	\end{align}
	where \\
	i) $x_{cr}^{-\epsilon}$ is a critical point of $f$ corresponding to the critical orbit of $K$ of type $-\epsilon$;\\
	ii) $\zeta_0 <x_{cr}^{-\epsilon}< \zeta_1$ are the coordinates of two consecutive points where $f^{''}$ vanishes, from both sides of $x_{cr}^{-\epsilon}$;\\
	iii) $x_0$ the point where $\gamma_0$ is tangent to $K$.\\
	Then for all the geodesics with $C_0^2=\epsilon f(x_0) \leq C^2 < M_{\epsilon}$, where $C_0$ is the Clairaut constant of $\gamma_0$, the Jacobi equation admits a non-periodic solution.
\end{proposition}
\begin{proof} 
Inequality (\ref{inégalité (2)}) implies that $(\omega_0 - \tau_0)^2 \epsilon \kappa(\omega_0) < \frac{\pi^2}{4}$ on $\gamma_0$, where $\tau_0$ is the same as in Lemma \ref{lemme inégalité (2)}, hence the Jacobi equation on $\gamma_0$ admits a non-periodic solution, by use of Lemma \ref{lemme inégalité (2)}.\\
This conclusion holds for $C^2 \geq C_0^2$, for $(\omega_C - \tau_C)$ and $\epsilon \kappa(\omega_C)$ (we put $C$ in index to say that it depends on the geodesic $\gamma_C$) are decreasing functions of $C$.
\end{proof}
\subsection{Stability by small deformation}\label{Section: Stability by small deformation}
\begin{lemma}\label{lemme géod près du bord}
	Let $g \in \mathcal{L}_K^*(T)$. Assume $\kappa$ vanishes $\mathfrak{n}$ times on the smallest period of $f$, hence twice between any two critical orbits of the Killing field of same type ($\mathfrak{n}$ is the number of zeros of $f$ in a period), and that these zeros are not on the extremums of $f$. Assume in addition that the zeros of $\kappa$ are simple zeros.  Let $\gamma_{\infty}$ be a critical orbit of $K$ corresponding to an extremum of $f$. Then there exists a neighborhood $V$ of the set $\{(g,\gamma_{\infty}(t), t \in \R)\}$ in $\mathcal{L}_K(T) \times T$, where $\mathcal{L}_K(T)$ is equipped with the $C^{\infty}$ topology, such that $\mathcal{Z} >0$ on $V \cap \Omega$. 
\end{lemma}
\begin{proof}
This amounts to saying that the geodesics near $\gamma_{\infty}$, for metrics close enough to $g$, where $g$ is the metric on the torus, are without conjugate points. Call $B$ the band containing $\gamma_{\infty}$, and fix $p$ a point on $\gamma_{\infty}$; it is sufficient to prove that there exists a neighborhood of $(g,p)$ in which $\mathcal{Z}$ is positive; the conclusion will follow from the compactness of $\gamma_{\infty}$.
So let $(g_n,p_n)$ be a sequence in $\Omega$ converging to $(g,p)$. For all $n$, denote by $\gamma_n$ the $g_n$-geodesic tangent to $K_n$ at $p_n$. Since $(g_n,p_n) \in \Omega$, we have $\forall n, C^2_n < \sup_B \epsilon g_n(K_n,K_n)$. Choose $p$ to be the origin on $T$ and denote by $x$ (resp. $x_n$) the transverse coordinate associated to $g$ (resp. $g_n$), with the fixed origin.  We can suppose that the points $p_n$ are all on the same side of the critical orbit of $K_n$ close to $\gamma_{\infty}$. Since the zeros of $\kappa$ are simple zeros, the curvature vanishes twice between two critical orbits of $K_n$ of same type for metrics sufficiently close to $g$. Set $x=x_0, x=x_1$ (resp. $x_0^n, x_1^n$) the smallest positive reals such that $g(K,K)$ (resp. $g_n(K_n,K_n)$) vanishes, and $x=\zeta_0, x=\zeta_1$ (resp. $\zeta_0^n, \zeta_1^n$) the smallest positive reals such that $\kappa$ (resp. $\kappa_n$) vanishes. When $C^2$ varies in $I_{\lambda}=[\sup(\epsilon f(\zeta_0),\epsilon f(\zeta_1))+\lambda,\sup_B \space \epsilon g(K,K)]$, $\lambda>0$ small enough, the $g$-geodesic  $\gamma_C$ cuts the orbits $x=\zeta_0$ and $x=\zeta_1$ of $K$. Now $f_n(\zeta_{0,1}^n)$ converges to $f(\zeta_{0,1})$, and $\sup_B \space \epsilon g_n(K_n,K_n)$ converges to $\sup_B \space \epsilon g(K,K)$; this ensures the existence of $\Lambda >0$ such that for $n$ big enough, we have $C^2_n > \Lambda + \sup(\epsilon f_n(\zeta_0^n),\epsilon f_n(\zeta_1^n))$. Set $t=0$ at the point where $\gamma_n$ is tangent to $K_n$ in the torus, and denote by $t_0^n, t_1^n$ the smallest positive reals such that $\beta_n^2(t_0^n) = \beta_n^2(t_1^n) =C^2_n$, and $t_{\zeta_0}^n, t_{\zeta_1}^n$ the smallest positive reals such that $\kappa_n( t_{\zeta_0^n}) = \kappa_n (t_{\zeta_1^n})=0$. There are four cases: 
\begin{figure}[h!] 
	\labellist 
	\small\hair 2pt 
	\pinlabel $\gamma$ at 17.25 385.5 
	\pinlabel $\gamma_{\infty}$ at 135 392 
	\pinlabel $\zeta_{0}$  at 202 386 
	\pinlabel $\zeta_{1}$ at 692 393 
	\pinlabel $x_0$  at 306 386
	\pinlabel $x_1$  at 606 386 
	\endlabellist 
	\centering 
	\includegraphics[scale=0.235]{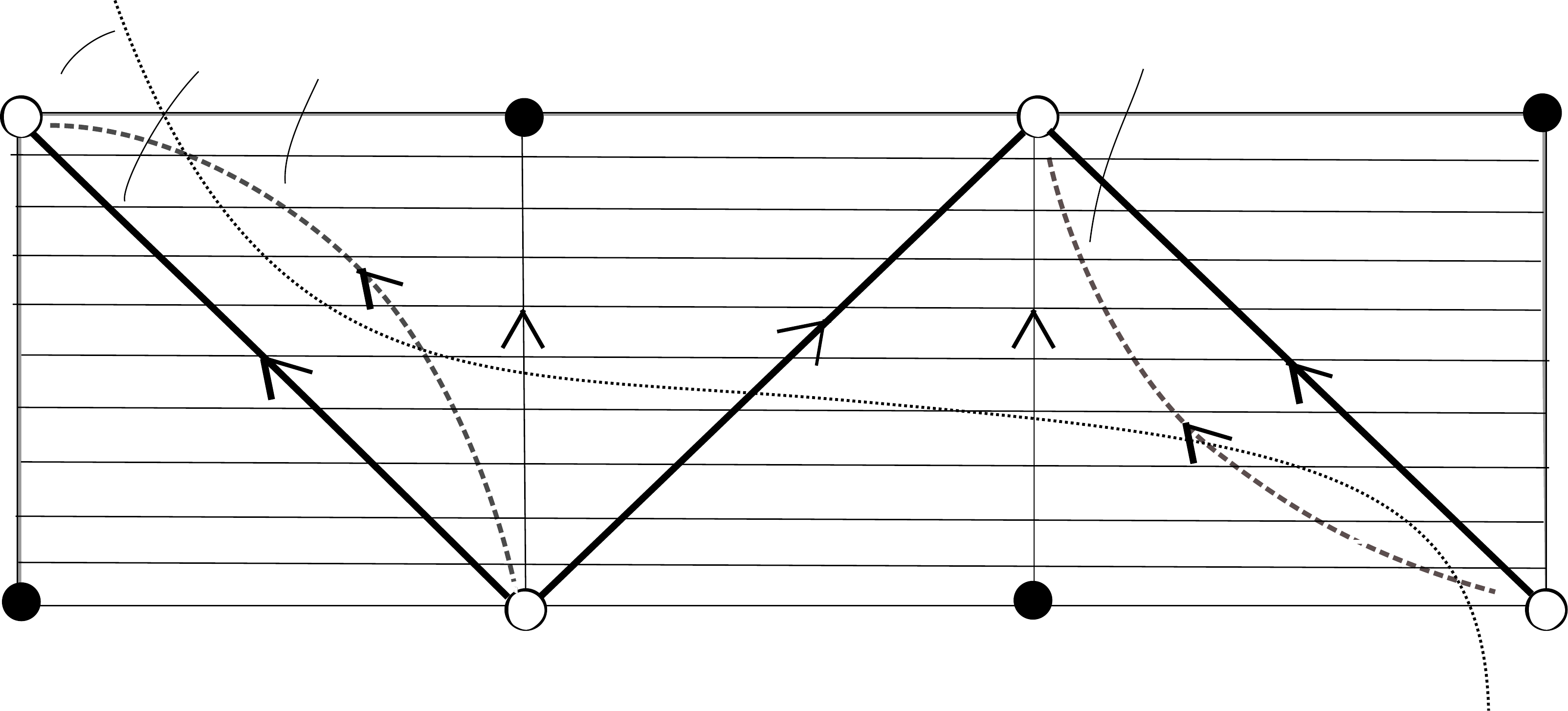} \caption{$\zeta_0 < x_0 < x_1 < \zeta_1$, i.e. \;\;$t_{\zeta_0^n} < t_0^n < \omega_n < t_1^n < t_{\zeta_1^n}$} 
	\label{fig:cobo}
\end{figure}
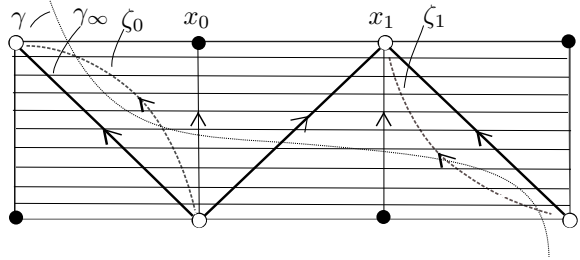 
\begin{figure}[h!] 
	\labellist 
	\small\hair 2pt 
	\pinlabel $\gamma$ at 12.75 387 
	\pinlabel $\gamma_{\infty}$ at 138.75 390 
	\pinlabel $\zeta_{0}$  at 202 387 
	\pinlabel $\zeta_{1}$ at 516 391 
	\pinlabel $x_0$  at 306 386
	\pinlabel $x_1$  at 606 386 
	\endlabellist 
	\centering 
	\includegraphics[scale=0.235]{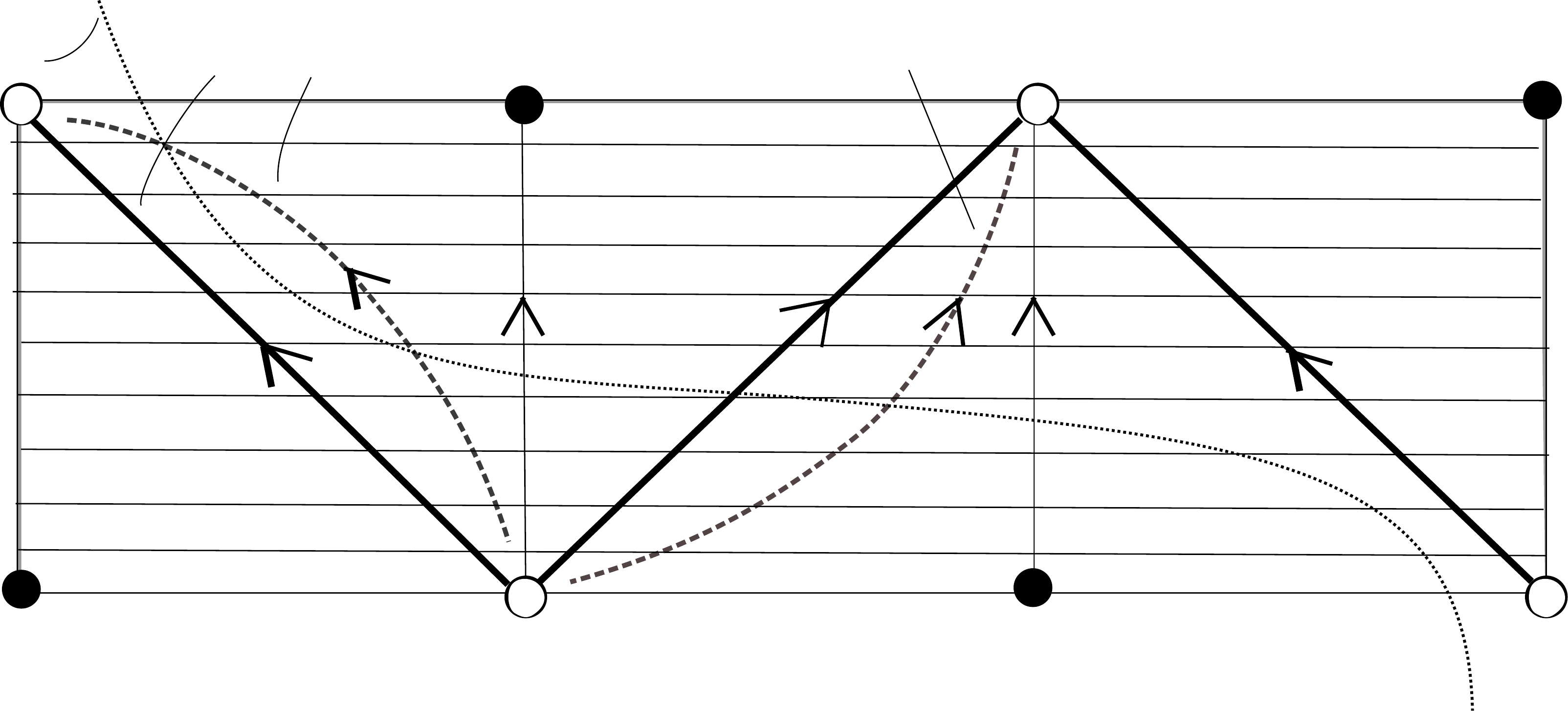} \caption{$\zeta_0 < x_0 < \zeta_1 < x_1$, i.e. \;\;$t_{\zeta_0^n} < t_0^n < \omega_n < t_{\zeta_1^n} < t_1^n$} 
	\label{fig:cobo}
\end{figure} 
\begin{figure}[h!] 
	\labellist 
	\small\hair 2pt 
	\pinlabel $\gamma$ at 13.5 382 
	\pinlabel $\gamma_{\infty}$ at 133.5 393 
	\pinlabel $\zeta_{0}$  at 407.25 388.5 
	\pinlabel $\zeta_{1}$ at 679 389.25 
	\pinlabel $x_0$  at 309 388
	\pinlabel $x_1$  at 611 388
	\endlabellist 
	\centering 
	\includegraphics[scale=0.235]{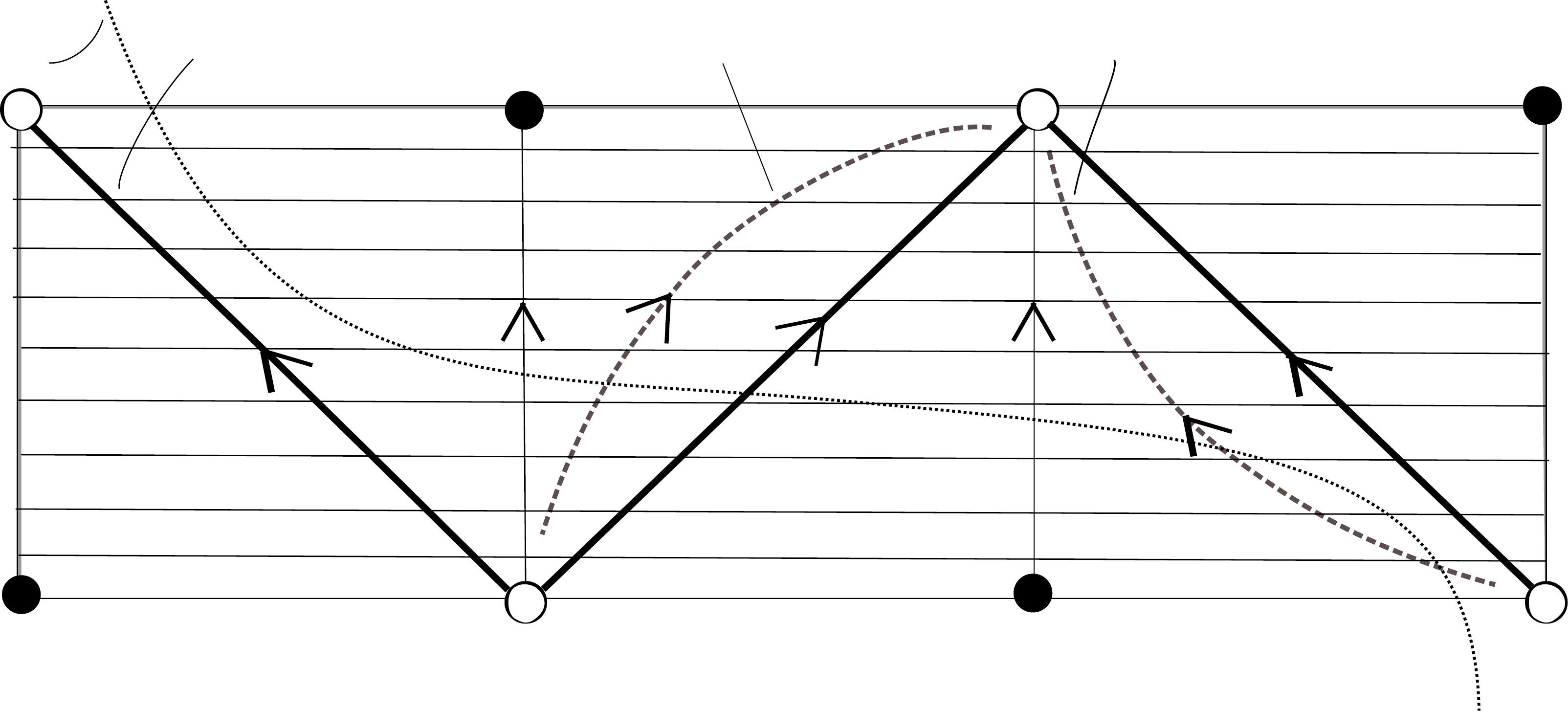} \caption{$x_0 < \zeta_0 < x_1 < \zeta_1$, i.e. \;\;$t_0^n < t_{\zeta_0^n} < \omega_n < t_1^n < t_{\zeta_1^n}$} 
	\label{fig:cobo}
\end{figure} 
\begin{figure}[h!] 
	\labellist 
	\small\hair 2pt 
	\pinlabel $\gamma$ at 12.75 382 
	\pinlabel $\gamma_{\infty}$ at 143.25 383.25 
	\pinlabel $\zeta_{0}$  at 470.25 386.25 
	\pinlabel $\zeta_{1}$ at 532.5 386.25 
	\pinlabel $x_0$  at 307.5 382
	\pinlabel $x_1$  at 611 382
	\endlabellist 
	\centering 
	\includegraphics[scale=0.235]{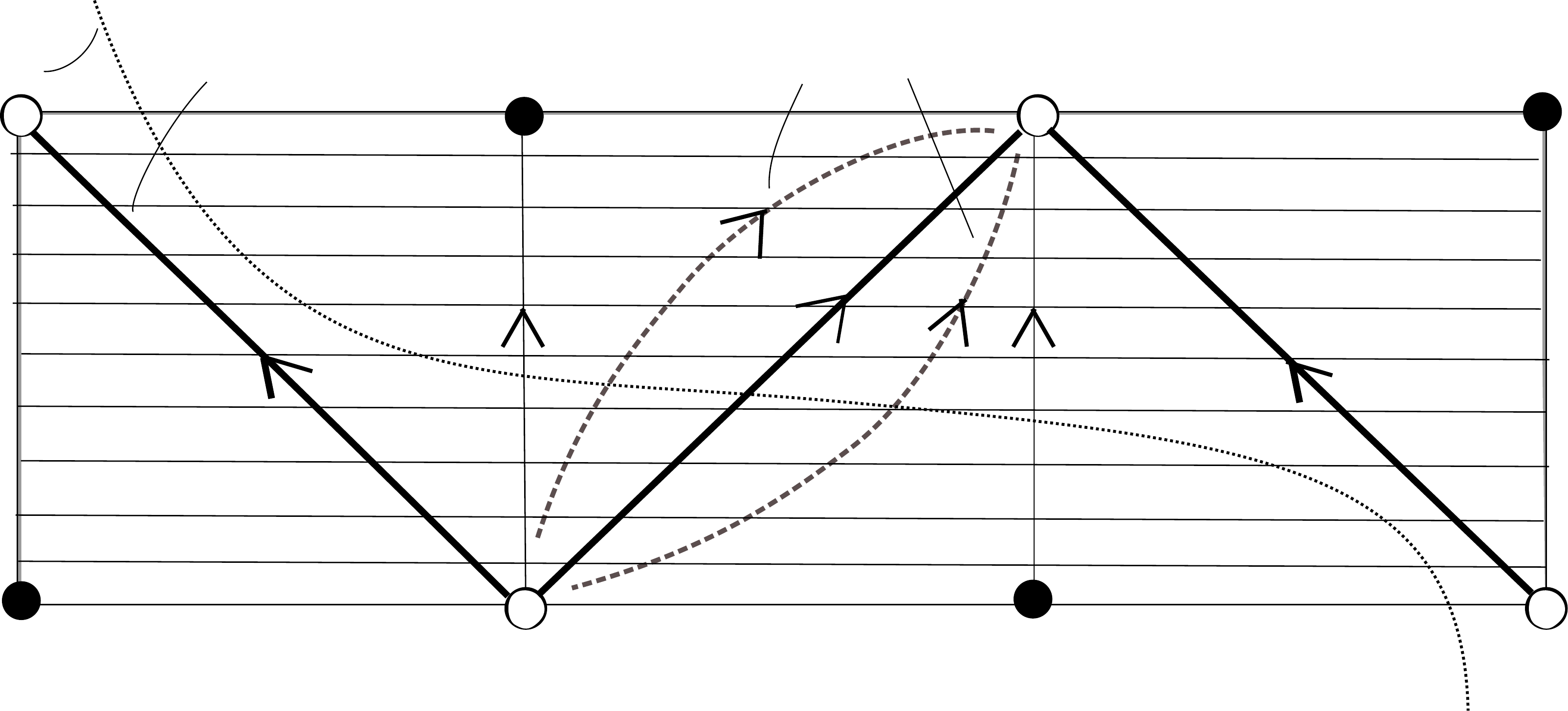} \caption{$x_0 \leq \zeta_0 < \zeta_1 \leq x_1$, i.e. \;\;$t_0^n \leq t_{\zeta_0^n} < \omega_n < t_{\zeta_1^n} \leq t_1^n$} 
	\label{fig:cobo}
\end{figure} 
\\
We first prove that the two sequences $(t_{\zeta_1^n} - t_{\zeta_0^n})_n$ and $(t_1^n - t_0^n)_n$ are bounded. 	
For this purpose, write $$t_{\zeta_1^n} - t_{\zeta_0^n} = \int_{\zeta_0^n}^{\zeta_1^n} \frac{1}{\sqrt{C^2_n - \epsilon f_n(x_n)}}\, dx_n.$$
Let $\mathcal{N}$ be a neighborhood of $(g,p)$ in $\mathcal{L}_K(T)$x$T$ in which we have $C^2_n > \Lambda + \sup(\epsilon f_n(\zeta_0^n),\epsilon f_n(\zeta_1^n))$.
For all $x \in [\zeta_0,\zeta_1]$, we have $$C^2 - \epsilon f(x) \geq \inf \{C^2 - \epsilon f(\zeta_0), C^2- \epsilon f(\zeta_1)\}=C^2 - \sup \{\epsilon f(\zeta_0), \epsilon f(\zeta_1)\},$$
hence, for metrics in $\mathcal{N}$, we get
\begin{align*}
t_{\zeta_1^n} - t_{\zeta_0^n} = \int_{\zeta_0^n}^{\zeta_1^n} \frac{1}{\sqrt{C^2_n - \epsilon f_n(x_n)}}\, dx_n &\leq  \int_{\zeta_0^n}^{\zeta_1^n} \frac{1}{\sqrt{C^2_n - \sup (\epsilon f_n(\zeta_0^n),\epsilon f_n(\zeta_1^n))}}\, dx_n\\
& \leq (\zeta_1^n - \zeta_0^n) \frac{1}{\sqrt{\Lambda}}. 
\end{align*}
Restricting to a subset of $\mathcal{N}$ if necessary, we can find a constant $A >0$ such that $t_{\zeta_1^n} - t_{\zeta_0^n} \leq A$ on this subset.  It follows that when $n$ tends to infinity, the difference $t_{\zeta_0^1} - t_{\zeta_0^n}$ is bounded, even if $t_{\zeta_0^n}$ and $t_{\zeta_1^n}$ go to infinity. Similarly, write 
$$ t_1^n - t_0^n = \int_{x_0^n}^{x_1^n} \frac{1}{\sqrt{C^2_n - \epsilon f_n(x_n)}}\, dx_n \leq (x_1^n - x_0^n) \frac{1}{\sqrt{C_n^2}}.$$
Using that $C_n$ is bounded away from zero establishes that $t_1^n - t_0^n$ is bounded, which is our assertion.  

Our next claim is that 	$c_n(t_{\zeta_1^n}) > 0$, provided $n$ is large enough. The solution $c_n$ is convex for $t \in [0,t_{\zeta_0^n}]$. Denote by $\tau_n$ and $\theta_n$ the smallest positive reals such that $c_n(\tau_n)=0$ and $c'_n(\theta_n)=0$. Suppose the claim were false. Then we could find a sequence $(n_i)_i \to \infty$ such that $\forall i$,\,  $\tau_{n_i} \leq \zeta_1^{n_i}$. We have $\beta_{n_i}(t)=\beta'_{n_i}(0) s_{n_i}(t)$; write  $\epsilon \beta^{'}(t)=\frac{1}{2}f^{'}(x(\gamma(t)))$, it appears that $\beta_{n_i}^{'}(0) \to 0$ when $i \to \infty$. Using $\langle K_n, K_n \rangle= \epsilon(C^2_n - \beta^2_n)$, we get 
$$s_{n_i}(t_{\zeta_0^{n_i}}) = \frac{C_{n_i}^2-\epsilon f_{n_i}(\zeta_0^{n_i})}{\beta'_{n_i}(0)}  \text{\;\;and\;\;} s_{n_i}(t_{\zeta_1^{n_i}}) = \frac{C^2_{n_i} - \epsilon f_{n_i}(\zeta_1^{n_i})}{\beta'_{n_i}(0)}$$ 
that tend to infinity when $i$ goes to infinity. It follows that  $s_{n_i}(\tau_{n_i}) \to \infty$; using $c_n s'_n - c'_n s_n =1$, this yields $c'_{n_i}(\tau_{n_i}) \to 0$. Now, since $\kappa_{n_i}$ has constant sign between $\theta_{n_i}$ and $\zeta_1^{n_i}$, $|c'_{n_i}(t)|$ reaches its maximal value at $t=\tau_{n_i}$; this gives for all $t \in [\theta_{n_i},\zeta_1^{n_i}]$, $|c'_{n_i}(t)| \leq |c'_{n_i}(\tau_{n_i})| $. From the mean value theorem, there exists $\mathfrak{a}_i \in [\theta_{n_i},\zeta_1^{n_i}]$  such that $c_{n_i}(\theta_{n_i}) - c_{n_i}(\zeta_1^{n_i})=c'_{n_i}(\mathfrak{a}_i)(\theta_{n_i} - \zeta_1^{n_i})$. Since $\theta_{n_i} - \zeta_1^{n_i}$ is uniformly bounded, we can make $c_{n_i}(\theta_{n_i}) - c_{n_i}(\zeta_1^{n_i})$ the smallest possible by setting $i$ large enough. Now from $c_{n_i}(\theta_{n_i}) > c_{n_i}(0)=1$, we see that provided $i$ is large enough, $c_{n_i}(\zeta_1^{n_i}) >0$, so $c_{n_i}$ does not vanish on $[0,\zeta_1^{n_i}]$, contrary to our assumption. This gives $c_n(t_0^n) >0$ for $n$ large enough, and in the cases (1) and (3), we also get  $c_n(t_1^n) >0$, making $c_n(t_0^n)+c_n(t_1^n)$ positive near the boundary.
\\ 
We are left with the cases (2) and (4). We already know that $c_n(t_{\zeta_1^n}) > 0$ near the boundary, and want to prove that $c_n (t_1^n) > 0$. We look at the set $A=\{c_n(t_1^n), c_n(t_1^n) < 0\}$. If  $A$ is finite, there is nothing to do; if it is infinite, denote its elements by $\{c_{n_i}(t_1^{n_i})\}$ and write $c_{n_i}(t_1^{n_i})-c_{n_i}(\tau_{n_i}) = c_{n_i}^{'}(\mathfrak{a}_i)(t_1^{n_i} - \tau_{n_i})$, where $\mathfrak{a}_i \in [\tau_{n_i}, t_1^{n_i}]$. We have $ 0 \geq c_{n_i}^{'}(t_1^{n_i})=\frac{c_{n_i}(t_1^{n_i})s_{n_i}^{'}(t_1^{n_i}) -1}{s_{n_i}(t_1^{n_i})} > -\frac{1}{s_{n_i}(t_1^{n_i})}$. Like before, $s_{n_i}(t_1^{n_i})$ goes to infinity when $i \to \infty$ , hence $c_{n_i}^{'}(t_1^{n_i}) \to 0$. Now, since $|c^{'}(t)| \leq |c_{n_i}^{'}(t_1^{n_i})|, \;\forall t \in [\tau_{n_i},t_1^{n_i}]$ and $t_1^{n_i} - \tau_{n_i}$ is bounded, we can make  $|c_{n_i}(t_1^{n_i})|$ the smallest possible. The proof is completed by observing that actually $c_n(t_0^n) >1$ for $n$ sufficiently large.  To do this, consider the set $B=\{c_n(t_0^n), t_0^n > \theta_n\}$. If $B$ is finite, the assertion follows, for $c_n$ is an increasing function on $[0,\theta_n]$. In the same manner, if $B$ is infinite, we exploit the fact that $c_n(t_0^n)$ is positive, and make it the closest possible to $c_n(\theta_{n_i})$. Since $c_n(\theta_{n_i}) >1$, the conclusion follows.   
 \end{proof}
\begin{theorem}\label{théorème de stabilité}
Let $(T,K)$ be a torus in the family given in Theorem \ref{famille SPC}. We exclude the case in which there are open sets of constant curvature. Assume $f$ satisfies the following conditions:\\
(1) $\kappa$ has simple zeros; \\
(2) There is only one critical orbit of $K$ in each band of the torus;\\
(3) For $\epsilon=\pm 1$, there exists $x_0$ in which $f$ has sign $\epsilon$, such that:	
$$-\epsilon f^{''}(x_0) \Big{(}\int_{x_0}^{x_1} \frac{dx}{\sqrt{M_{\epsilon} -\epsilon f(x)}}\Big{)}^2 < 2 \pi^2,$$
and
$$\epsilon f^{''}(x_{cr}^{-\epsilon}) \Big{(}\int_{\zeta_0}^{\zeta_1} \frac{dx}{\sqrt{\epsilon f(x_0) -\epsilon f(x)}}\Big{)}^2 < 2 \pi^2,$$
where $x_{cr}^{-\epsilon}$, $\zeta_0$, $\zeta_1$, $x_1$ and $M_{\epsilon}$ are as in propositions \ref{proposition inégalité (1)} and \ref{proposition inégalité (2)}.\\
Then, there is a neighborhood $\mathcal{N}$ of the torus in $\mathcal{L}_K(T)$ such that the metrics in $\mathcal{N} \cap S\mathcal{L}_K(T)$ have no conjugate points, whereas the others admit conjugate points. 

\end{theorem}
\begin{proof}
Denote by $g_0$ the metric on $T$; $g_0$ has no conjugate points. The assumptions on $g_0$ imply that a neighborhood of $g_0$ in $\mathcal{L}_K(T)$ can be taken in $\mathcal{L}_K^*(T)$. By propositions \ref{proposition inégalité (1)} and \ref{proposition inégalité (2)}, conditions (1) and (3) imply that the extended function $\mathcal{Z}(g_0)$ is strictly positive on the open set $$\widetilde{\Omega}(g_0)=\{p \in T, \; g_0(\nabla_{K_0} K_0,\nabla_{K_0} K_0)(p) \neq 0\}.$$ 
Therefore, there exists a neighborhood of any compact subset of $\widetilde{\Omega}(g_0)$ in $$\widetilde{\Omega}(g)=\{(g,p) \in S\mathcal{L}_K^*(T) \times T, \; g(\nabla_K K,\nabla_K K)(p) \neq 0\}$$ on which $\mathcal{Z}$ is positive. Thus, we get a neighborhood $V$ of $g_0$ in $\mathcal{L}_K^*(T)$ such that $\mathcal{Z}$ is positive on $SV \times F$, where $SV = V \cap S\mathcal{L}_K^*(T)$ and $F$ is a compact subset of $T$ not containing the critical orbits of $K_0$. We can make this compact subset the closest possible to the boundary. Now, the condition $f^{'}f^{'''} \leq 0$, combined with (1), implies that $\kappa$ has exactly two simple zeros in the smallest period of $f$. Besides, these zeros cannot be on the extremums of $f$ unless the torus is flat (see the proof of Lemma \ref{lemme inégalité (2)}). Therefore, Lemma \ref{lemme géod près du bord} provides a neighborhood $V^{'}$ of $g_0$ in $\mathcal{L}_K(T)$ such that $\mathcal{Z}$ is positive on $V^{'} \times U$, $U$ being a neighborhood of the critical orbits of $K_0$ in the torus. Taking $V \cap V^{'}$  achieves the proof.  
\end{proof}
Theorem \ref{théorème de stabilité} gives a way to obtain examples of Lorentzian metrics on $T$ with no conjugate points, that are stable by deformation in $S\mathcal{L}_K^*(T)$. Here are a few examples:

. The Clifton-Pohl torus, corresponding to $f(x)=sin(2x)$;\\
. $f(x)= \frac{sin(x)}{10+sin(x)}$;\\
. $f(x)=ln(2+sin x)$;\\
. $f(x)=cos(sin (x)) -3/4$;\\
. $f(x)=JacobiSD(x,1/2)$;\\
. $f(x)= JacobiSN(x,1/4)$.\\
The verification of hypothesis (3) of Theorem \ref{théorème de stabilité} is done numerically.
\begin{remark}
	The quadratic variations of the Clifton-Pohl torus are also stable by deformation in $S\mathcal{L}^*_K(T)$; this comes from an explicit resolution of the Jacobi equation.
\end{remark}

\end{document}